\documentclass[a4paper]{article}

\usepackage{exscale}
\usepackage{amsmath,amssymb,mathtools}
\usepackage{amsthm}
\usepackage{ifthen}
\usepackage{tikz}
\usepackage{graphicx}
\usepackage{cite}
\usepackage{caption}
\DeclareCaptionFormat{empty}{#1}

\newtheorem{prop}{Proposition}[section]
\newtheorem{lemma}[prop]{Lemma}
\newtheorem{theo}[prop]{Theorem}
\newtheorem{coroll}[prop]{Corollary}
\theoremstyle{definition}
 \newtheorem{remark}[prop]{Remark}
 
 \newtheorem{example}[prop]{Example}

\newcommand{\pmat}[1]{\begin{pmatrix}#1\end{pmatrix}}
\newcommand{\set}[2]{\left\{#1\,\middle|\,#2\right\}}

\newcommand{\iprod}[2]{\langle#1,#2\rangle}
\newcommand{\bigiprod}[2]{\left\langle#1,#2\right\rangle}

\DeclareMathOperator{\Real}{Re}

\DeclareMathOperator{\dist}{dist}

\DeclareMathOperator{\diag}{diag}

\newcommand{\N}{\mathbb{N}}

\newcommand{\R}{\mathbb{R}}
\newcommand{\C}{\mathbb{C}}
\newcommand{\bbP}{\mathbb{P}}
\newcommand{\mdef}{\mathcal{D}}
\newcommand{\range}{\mathcal{R}}
\newcommand{\eps}{\varepsilon}
\newcommand{\e}{\mathrm{e}}
\newcommand{\iu}{\mathrm{i}}
\newcommand{\D}{\mdef}
\newcommand{\mlin}{\mathcal{L}}
\newcommand{\A}{\mathcal{A}}
\newcommand{\cP}{\mathcal{P}}
\newcommand{\sigmaapp}{\sigma_{\mathrm{app}}}

\newcommand{\lv}[1]{\textcolor{black}{#1}}
\newcommand{\cw}[1]{\textcolor{black}{#1}}
\newcommand{\bj}[1]{\textcolor{black}{#1}}

\title{Pseudospectrum enclosures by discretization}

\author{Andreas Frommer, Birgit Jacob, Lukas Vorberg,\\ Christian Wyss and  Ian Zwaan\thanks{University of Wuppertal, Faculty of Mathematics and Natural Sciences, Department of Mathematics and Computer Sciences, IMACM, D-42097 Wuppertal, Germany 
  (frommer@math.uni-wuppertal.de, jacob@math.uni-wuppertal.de, vorberg@uni-wuppertal.de, wyss@math.uni-wuppertal.de, zwaan@math.uni-wuppertal.de).
  }
}

\date{}

\begin{document}
\maketitle

\begin{abstract}
  A new method to enclose the pseudospectrum
  via the numerical range of the inverse of a matrix or linear operator
  is presented. 
  The method is applied to finite-dimensional
  discretizations of an  operator on an infinite-dimensional Hilbert space,
  and convergence results for different approximation schemes are obtained,
  including finite element methods.
  We show that the pseudospectrum of the full operator
  is contained in an intersection of sets which are expressed in terms
  of the numerical ranges of shifted inverses of the 
  approximating matrices.
  The results are illustrated by means of two examples:
  the advection-diffusion operator and the Hain-L\"ust operator.
\end{abstract}

\section{Introduction}

Traditional stability analysis of linear dynamic models is based on eigenvalues. Thus 
determining the eigenvalues of a  matrix or, more generally, the spectrum of a 
linear operator is a major task in analysis and numerics. The explicit computation of the whole spectrum of a linear operator by analytical or numerical techniques
is only possible  in rare cases. Moreover, the spectrum is in general quite sensitive
with respect to small perturbations of the operator. This is in particular true for non-normal matrices and operators.
 Therefore, one is interested in supersets of the spectrum that are easier to compute
 and that are also robust under perturbations. One suitable superset is the $\varepsilon$-pseudospectrum, a notion which has been independently introduced by Landau \cite{Lan75}, Varah \cite{Var79}, Godunov \cite{KR85}, Trefethen \cite{Tre90} and Hinrichsen and Pritchard \cite{HP92}. The $\varepsilon$-pseudospectrum of a linear operator $A$ on a Hilbert space $H$ consists of the union of the spectra of all operators on $H$ of the form $A+P$ with $\|P\|<\varepsilon$.
 Besides the fact that the pseudospectrum is robust under perturbations,
 it is also suitable to determine the transient growth behavior of linear dynamic models in finite time, which may be far from the asymptotic behavior.
For an overview on the pseudospectrum and its applications we refer the reader to \cite{ETBook} and \cite{gateway}.

Numerical computation of the pseudospectrum of a matrix has been 
intensively studied in the literature. Most algorithms use   simple grid-based methods,
where one computes the smallest singular value of $A-z$
at the points $z$ of a grid, or 
path-following methods, see the survey  \cite{Tre99a} or the overview at \cite{gateway}. 
Both methods face several challenges.
The main problem of grid-based methods is first to find a suitable region in the
complex plane and then to perform the computation on a usually very
large number of grid points.
The main difficulty of path-following algorithms is to find a starting point, that is, a point on the boundary of the pseudospectrum.
Moreover, as the pseudospectrum may be disconnected it is difficult to find every component. 
However, there are several speedup techniques available, see  \cite{Tre99a}, which are essential for applications. 


\lv{A simple method to enclose the pseudospectrum is in terms of the numerical range. More precisely, under an additional weak assumption, the $\eps$-pseudospectrum is contained in an $\eps$-neighborhood of the numerical range of the operator, see Remark \ref{rem:NumericalRange}. While this superset is easy to compute for matrices, it can not distinguish disconnected components of the pseudospectrum as the numerical range is convex.}

In this article we propose a new method to enclose the pseudospectrum via the numerical range of the inverse of the matrix or linear operator.  
More precisely, for a linear operator $A$ on a Hilbert space and $\varepsilon>0$ we show 
\begin{equation}\label{eqn:enclusion}
  \sigma_\eps(A)\subset\bigcap_{s\in S} \left[\bigl(B_{\delta_s}(W((A-s)^{-1}))\bigr)^{-1}+s\right],
\end{equation}
see Theorem~\ref{theo:psincl-shift}.
Here $\sigma_\eps(A)$ denotes the $\eps$-pseudospectrum of $A$,
$W((A-s)^{-1})$ is the numerical range of the resolvent operator $(A-s)^{-1}$,
$B_{\delta_s}(U)$ is the ${\delta_s}$-neighborhood of a set $U$,
and $S$ is a suitable subset of the complex plane.
This inclusion holds for matrices as well as for linear operators on
Hilbert spaces. \bj{Further, we show that the enclosure of the
  pseudospectrum in \eqref{eqn:enclusion} becomes optimal if the set $S$ is chosen optimally, see Theorem \ref{theo:optimality}.} The idea to study the numerical range of the inverses stems from the fact that the spectrum of a matrix can be expressed in terms of inverses of shifted matrices  \cite{Hochstenbach}.

From a numerical point of view this new method faces
similar challenges as grid-based methods as a suitable set $S$ of
points has to be found and then the numerical ranges of a large number of
matrices have to be computed.
However, this new method has the advantage that it enables us to
enclose the pseudospectrum of an infinite-dimensional operator by a
set which is expressed by the approximating matrices.

The usual procedure to compute the pseudospectrum of a linear operator
on an infinite-dimensional Hilbert space is to approximate it by
matrices and then to calculate the pseudospectrum of one of the
approximating matrices. In \cite[Chapter 43]{ETBook} spectral methods
are used for the approximation, but no convergence properties of the
pseudospectrum under discretization are proved. So far only few
results are available concerning the relations between the
pseudospectra of the discretized operator and those of the
infinite-dimensional operator.
Convergence properties of the pseudospectrum under discretization have
been studied for the linearized Navier-Stokes equation \cite{Gerecht},
for band-dominated bounded operators \cite{Lindner} and for Toeplitz
operators \cite{Boettcher}.
B\"ogli and Siegl \cite{boegli18,boegli-siegl14}
prove local and global convergence of the pseudospectra
of a sequence of linear operators which converge in a generalized resolvent sense.
Further, Wolff \cite{wolff} shows some
abstract convergence results for the approximate point spectrum of a
linear operator using the pseudospectra of the approximations.

In this article we refine the enclosure \eqref{eqn:enclusion} of the
pseudospectrum of linear operators further
and show that it is
sufficient to calculate the numerical ranges of approximating
matrices.
More precisely,
we show in Theorem~\ref{theo:psincl-str-shift} that
\begin{equation}\label{eqn:enclusion2}
  \sigma_\eps(A)\subset\bigcap_{s\in S} \left[\bigl(B_{\delta_s}(W((A_n-s)^{-1}))\bigr)^{-1}+s\right]
\end{equation}
if  $n$ is sufficiently large.
Here $A_n$ is a sequence of matrices which approximates the operator
$A$ strongly. We refer to Section \ref{sec:strappr} for the precise
definition of strong approximation. If we even have a uniform
approximation of the operator $A$, then we are able to prove an
estimate for the index $n$ such that \eqref{eqn:enclusion2} holds in
intersections with compact subsets of the complex plane,
see Section~\ref{sec:unifap}.
In Section~\ref{sec:fe-discr} we show that
finite element discretizations of elliptic partial differential
operators yield uniform approximations. Further, as an example of
strong approximation we study in Section~\ref{sec:discr-block}
a class of structured block operator
matrices.  In the final section we apply our obtained results to the
advection-diffusion operator and the Hain-L\"ust operator.

We conclude this introduction with some remarks on the notation used.
Let $H$ be a Hilbert space. Throughout this article we assume that
$A:\D(A)\subset H\rightarrow H$ is a closed, densely defined,
linear operator.
We denote the range of $A$ by $\range(A)$
and the spectrum by $\sigma(A)$.
The resolvent set is $\varrho(A)=\C\backslash \sigma(A)$.
Let $\mlin(H_1, H_2)$ denote the set of
linear, bounded operators from the Hilbert space $H_1$ to the Hilbert
space $H_2$.
The operator norm of $T\in \mlin(H_1, H_2)$ will be
denoted by $\|T\|_{\mlin(H_1, H_2)}$.
To shorten notation, we write
$\mlin(H)=\mlin(H, H)$ and denote the operator norm of $T\in\mlin(H)$
by $\|T\|$.
The identity operator is denoted by $I$.
For every
$\lambda\in\varrho(A)$, the resolvent $(A-\lambda)^{-1}:=(A-\lambda I)^{-1}$ satisfies
$(A-\lambda)^{-1}\in\mlin(H)$.
For a set of complex numbers $S\subset\C$
we denote the $\delta$-neighborhood by $B_\delta(S)$, i.e.,
\(B_\delta(S)=\set{z\in\C}{\dist(z,S)<\delta}\),
and we also use the notation
\(S^{-1}=\set{z^{-1}}{z\in S\setminus\{0\}}\). Further, we use the notation $\mathbb C^*:=\mathbb C\backslash\{0\}$.

\section{Pseudospectrum enclosures using the numerical range}
\label{sec:psincl}

In this section
we present the basic idea of considering numerical ranges of
shifted inverses of an operator in order to obtain an
enclosure of its pseudospectrum.
We start by recalling the notions of the numerical range and the
$\eps$-pseudo\-spectrum.

The \emph{numerical range} of an operator $A$ is defined as the set
\[W(A)=\set{\iprod{Ax}{x}}{x\in\mdef(A),\,\|x\|=1},\]
see e.g.\ \cite{Kato95}.
It is always a convex set and, if $A$ is additionally bounded, then $W(A)$
is bounded too.
\cw{The \emph{numerical radius} is $w(A)=\sup\set{|z|}{z\in W(A)}$.}
The numerical range satisfies  the inclusions
\[\sigma_p(A)\subset W(A),\qquad \sigmaapp(A)\subset\overline{W(A)},\]
where $\sigma_p(A)$ is the point spectrum of $A$, i.e., the set of all
eigenvalues and $\sigmaapp(A)$ is the so-called
approximate point spectrum defined by
\[\sigmaapp(A)=\set{\lambda\in\C}{\exists x_n\in\mdef(A),\,
  \|x_n\|=1 : \lim_{n\to\infty}(A-\lambda)x_n=0}.\]
The spectrum, point spectrum and approximate point spectrum are related by
$\sigma_p(A)\subset \sigmaapp(A)\subset\sigma(A)$.
If $A$ has a compact resolvent, then the spectrum consists of eigenvalues
only and hence we have equality.

For $\eps>0$ the \emph{$\eps$-pseudospectrum} of $A$ is given by
\[\sigma_\eps(A)=\sigma(A)\cup\set{\lambda\in\varrho(A)}{\|(A-\lambda)^{-1}\|>\frac{1}{\eps}}.\]
If we understand $\|(A-\lambda)^{-1}\|$ to be infinity for $\lambda\in\sigma(A)$,
then this can be shortened to
\[\sigma_\eps(A)=\set{\lambda\in\C}{\|(A-\lambda)^{-1}\|>\frac{1}{\eps}}.\]
Hence
\[\C\setminus\sigma_\eps(A)=\set{\lambda\in\varrho(A)}{\|(A-\lambda)^{-1}\|\leq\frac{1}{\eps}}.\]

The central idea of this article is the following:
If $\lambda\in\C$ is such that $1/\lambda$ has a certain positive distance $\delta$
to the numerical range
of the inverse operator $A^{-1}$, then this yields an estimate of the
form
\[\|(A-\lambda)x\|\geq \eps\|x\|,\qquad x\in\mdef(A),\]
with some constant $\eps>0$, which will in turn be used to show $\lambda\in\varrho(A)$
with $\|(A-\lambda)^{-1}\|\leq\frac{1}{\eps}$, i.e., $\lambda\not\in\sigma_\eps(A)$.
This is made explicit with the next proposition:
\begin{prop}\label{prop:psincl}
  Suppose that $0\in \varrho(A)$. Then for every \lv{$0<\eps <  \frac{1}{\|A^{-1}\|}$} and
  \lv{$\delta=\frac{\|A^{-1}\|^2\eps}{1-\|A^{-1}\|\eps}$} we have
  \[\sigma_\eps(A)\subset\bigl(B_\delta(W(A^{-1}))\bigr)^{-1}.\]
\end{prop}
\begin{proof}
  Let us denote
  $U=\bigl(B_\delta(W(A^{-1}))\bigr)^{-1}$.
  As a first step we show that
  \begin{equation}\label{eq:regulartype}
    \|(A-\lambda)x\|\geq\eps \quad\text{for all}\quad
    \lambda\in\C\setminus U,\,x\in\mdef(A),\,\|x\|=1.
  \end{equation}
  So let $\lambda\in \C\setminus U$.
  We consider two cases. First suppose that \lv{$|\lambda|>\frac{1}{\|A^{-1}\|}-\eps$}.
  Then $\lambda\neq0$, $\lambda^{-1}\not\in B_{\delta}(W(A^{-1}))$ and hence
  $\dist(\lambda^{-1},W(A^{-1}))\geq\delta$.
  For $x\in\mdef(A)$, $\|x\|=1$ we find
  \[\delta\leq|\lambda^{-1}-\iprod{A^{-1}x}{x}|
  =|\iprod{(\lambda^{-1}-A^{-1})x}{x}|
  \leq\|(\lambda^{-1}-A^{-1})x\|.\]
  Consequently
  \begin{align*}
    \|(A-\lambda)x\|&=|\lambda|\|A(\lambda^{-1}-A^{-1})x\|
    \geq\frac{|\lambda|}{\|A^{-1}\|}\|(\lambda^{-1}-A^{-1})x\|\\
    &\geq\lv{\frac{\delta}{\|A^{-1}\|}\left(\frac{1}{\|A^{-1}\|}-\eps\right)
    = \frac{\delta(1-\|A^{-1}\|\eps)}{\|A^{-1}\|^2}}=\eps.
  \end{align*}
  In the other case if \lv{$|\lambda|\leq\frac{1}{\|A^{-1}\|}-\eps$
  then $|\lambda|\|A^{-1}\|\leq 1-\|A^{-1}\|\eps$ and hence} $I-\lambda A^{-1}$ is invertible by a Neumann series argument
  with \lv{$\|(I-\lambda A^{-1})^{-1}\|\leq\frac{1}{\|A^{-1}\|\eps}$}.
  For $x\in\mdef(A)$, $\|x\|=1$ this implies
  \[\|(A-\lambda)x\|=\|A(I-\lambda A^{-1})x\|
  \geq\frac{1}{\|A^{-1}\|\|(I-\lambda A^{-1})^{-1}\|}\geq\eps.\]
  We have thus shown \eqref{eq:regulartype}.
  In particular, $\lambda\in \C\setminus U$ implies $\lambda\not\in\sigmaapp(A)$, i.e.,
  \begin{equation}\label{eq:appdisj}
    \sigmaapp(A)\cap\C\setminus U=\varnothing.
  \end{equation}
  Since $B_\delta(W(A^{-1}))$ is convex and bounded,
  the set $\C^*\setminus B_\delta(W(A^{-1}))$ is connected and hence also
  \[\C^*\setminus U=\left(\C^*\setminus B_\delta(W(A^{-1}))\right)^{-1},\]
  the image under the homeomorphism $\C^*\to\C^*$, $z\mapsto z^{-1}$.
  On the other hand, the boundedness of $B_\delta(W(A^{-1}))$ implies that
  a neighborhood around 0 belongs to
  $\C\setminus U=(\C^*\setminus U) \cup \{0\}$.
  Consequently, the set $\C\setminus U$ is connected and satisfies
  $0\in\varrho(A)\cap\C\setminus U$.
  Using \eqref{eq:appdisj} and the fact that
  $\partial \sigma(A)\subset\sigmaapp(A)$, we conclude that
  \[\C\setminus U\subset\varrho(A).\]
  Here $\partial \sigma(A)$ denotes the boundary of the spectrum of $A$.
  Now \eqref{eq:regulartype} implies that if $\lambda\in\C\setminus U$ then
  $\|(A-\lambda)^{-1}\|\leq\frac{1}{\eps}$ 
  and therefore we obtain $\lambda\not\in\sigma_\eps(A)$.
\end{proof}

Applying the last result to the shifted operator $A-s$ and then taking
the intersection over a suitable set of shifts, we obtain
our first main result on an enclosure of the pseudospectrum:
\begin{theo}\label{theo:psincl-shift}
  Consider a set $S\subset\varrho(A)$ such that
  \[M:=\sup_{s\in S}\|(A-s)^{-1}\|<\infty.\]
  Then for \lv{$0<\eps < \frac{1}{M}$} we get the inclusion
  \begin{align}\label{eq:psincl-shift}
  \sigma_\eps(A)\subset\bigcap_{s\in S}
  \left[\bigl(B_{\delta_s}(W((A-s)^{-1}))\bigr)^{-1}+s\right]
  \end{align}
  where \lv{$\delta_s=\frac{\|(A-s)^{-1}\|^2\eps}{1-\|(A-s)^{-1}\|\eps}$}.
\end{theo}
\begin{proof}
  For every $s\in S$ we can apply Proposition~\ref{prop:psincl}
  to the operator $A-s$ and obtain
  \[\sigma_\eps(A)-s=\sigma_\eps(A-s)\subset\bigl(B_{\delta_s}(W((A-s)^{-1}))\bigr)^{-1}.\qedhere \]
\end{proof}

The following simple example demonstrates
that the
$\delta$-neighborhood around the numerical range
is actually needed to obtain an enclosure of the
pseudospectrum.

\begin{example}
Let $A=\diag(-1+i,-1-i,1+i,1-i)\in\C^{4\times 4}$.
Then $A^{-1}=\frac12\diag(-1-i,-1+i,1-i,1+i)$.
Since $A^{-1}$ is normal, its numerical range is simply the convex hull of
its eigenvalues.
Thus $W(A^{-1})$ is the following square:
\begin{center}
  \begin{tikzpicture}
    \fill[lightgray] (-.5,-.5) rectangle(.5,.5);
    \draw (.5,.5)--(-.5,.5)--(-.5,-.5)--(.5,-.5)--(.5,.5);
    \fill (.5,.5) circle(2pt);
    \fill (-.5,.5) circle(2pt);
    \fill (-.5,-.5) circle(2pt);
    \fill (.5,-.5) circle(2pt);

    \draw (1,-2pt)--(1,2pt);
    \draw (1,-.3) node{$1$};
    \draw (-1,-2pt)--(-1,2pt);
    \draw (-1,-.3) node{$-1$};
    \draw (-2pt,1)--(2pt,1);
    \draw (-.2,1) node{$1$};
    \draw (-2pt,-1)--(2pt,-1);
    \draw (-.4,-1) node{$-1$};

    \draw[->] (-2,0)--(2,0);
    \draw[->] (0,-2)--(0,2);
  \end{tikzpicture}
\end{center}

Then, using the fact that $z\mapsto\frac{1}{z}$ is a M\"obius transformation,
we obtain for $W(A^{-1})^{-1}$ the following curve plus its exterior:
\begin{center}
  \begin{tikzpicture}
    \filldraw[fill=lightgray,draw=lightgray] (2.5,-2.5)--(1,-1) arc(-90:90:1) arc(0:180:1)
    --(-2.5,2.5)--(2.5,2.5)--cycle;
    \fill[lightgray] (-2.5,2.5)--(-1,1) arc(90:270:1) arc(-180:0:1)
    --(2.5,-2.5)--(-2.5,-2.5)--cycle;
    
    \draw (1,-1) arc(-90:90:1);
    \draw (1,1) arc(0:180:1);
    \draw (-1,1) arc(90:270:1);
    \draw (-1,-1) arc(-180:0:1);
    \fill (1,1) circle(2pt);
    \fill (-1,1) circle(2pt);
    \fill (-1,-1) circle(2pt);
    \fill (1,-1) circle(2pt);

    \draw (1,-2pt)--(1,2pt);
    \draw (1,-.3) node{$1$};
    \draw (2,-2pt)--(2,2pt);
    \draw (2.2,-.3) node{$2$};
    \draw (-1,-2pt)--(-1,2pt);
    \draw (-1,-.3) node{$-1$};
    \draw (-2,-2pt)--(-2,2pt);
    \draw (-2.3,-.3) node{$-2$};

    \draw (-2pt,1)--(2pt,1);
    \draw (-.2,1) node{$1$};
    \draw (-2pt,-1)--(2pt,-1);
    \draw (-.4,-1) node{$-1$};

    \draw[->] (-3,0)--(3,0);
    \draw[->] (0,-3)--(0,3);
  \end{tikzpicture}
\end{center}
We see that $W(A^{-1})^{-1}$ touches the spectrum of $A$.
This is of course clear: if an eigenvalue $1/\lambda$ of $A^{-1}$ is on the boundary
of $W(A^{-1})$, then the eigenvalue $\lambda$ of $A$ is on the boundary
of $W(A^{-1})^{-1}$.
In particular in this example we do not have $\sigma_\eps(A)\subset W(A^{-1})^{-1}$
for any $\eps>0$ since $\sigma_\eps(A)$ contains discs with radius $\eps$ around the
eigenvalues.
\end{example}

\lv{
\begin{prop}\label{prop:rho}
  For $s\in\varrho(A)$, $0<\eps<\frac{1}{\|(A-s)^{-1}\|}$ and
  $\delta_s = \frac{\|(A-s)^{-1}\|^2\eps}{1-\|(A-s)^{-1}\|\eps}$ we have that
  \begin{align*}
    \overline{B_{\rho_s}(s)} \cap
    \left[\bigl(B_{\delta_{{s}}}(W((A-{s})^{-1}))\bigr)^{-1}+{s}\right]
    = \varnothing
  \end{align*}
  where
  \(\rho_s \cw{= \frac{1}{w((A-s)^{-1})+\delta_s}}
  \geq\frac{1}{\|(A-s)^{-1}\|+\delta_s}\).
\end{prop}
\begin{proof}
  Let $s\in \varrho(A)$ and $t\in \bigl(B_{\delta_s}(W((A-s)^{-1}))\bigr)^{-1}+s$. Then
  \begin{align*}
    \frac{1}{t-s}\in B_{\delta_s}(W((A-s)^{-1}))
  \end{align*}
  and we can estimate
  \begin{align*}
    \frac{1}{|t-s|} &< \delta_s + \sup_{\|x\|=1}|\langle(A-s)^{-1}x,x\rangle|
    \cw{=\delta_s+w((A-s)^{-1})}
    = \frac{1}{\rho_s}.
  \end{align*}
  This implies $|t-s|>\rho_s$ and therefore $t\notin\overline{B_{\rho_s}(s)}$.
\end{proof}
}

\lv{The following theorem shows that the enclosure of the
pseudospectrum in Theorem \ref{theo:psincl-shift} becomes optimal if the
shifts are chosen optimally.
}

\lv{
\begin{theo}\label{theo:optimality}
Let $\eps>0$, $S_\gamma = \set{s\in \varrho(A)}{\|(A-s)^{-1}\|=\frac{1}{\eps+\gamma}}$ for $\gamma>0$ and $\delta_s=\frac{\|(A-s)^{-1}\|^2\eps}{1-\|(A-s)^{-1}\|\eps}$. Then:
\begin{enumerate}
{\item
\abovedisplayskip=-\baselineskip
\belowdisplayskip=0pt
\abovedisplayshortskip=-\baselineskip
\belowdisplayshortskip=0pt
\begin{align*}
\sigma_\eps(A) &\subset \bigcap_{\gamma>0}\bigcap_{s\in S_\gamma}\left[\left(B_{\delta_s}(W((A-s)^{-1}))\right)^{-1}+s\right]\\
&\subset \set{\lambda\in \C}{\|(A-\lambda)^{-1}\|\geq \frac{1}{\eps}},
\end{align*}}
{\item
\abovedisplayskip=-\baselineskip
\belowdisplayskip=0pt
\abovedisplayshortskip=-\baselineskip
\belowdisplayshortskip=0pt
\begin{align*}
\set{\lambda\in\C}{\|(A-\lambda)^{-1}\|\geq\frac{1}{\eps}} = \bigcap_{\gamma>0}\bigcap_{s\in S_\gamma}\left[\left(\overline{B_{\delta_s}(W((A-s)^{-1}))}\right)^{-1}+s\right],
\end{align*}}
\item Under the additional assumption that $A$ is normal with compact resolvent and $L>0$, there exists an $\eps_0>0$ such that for all $\eps<\eps_0$
\begin{align*}
\sigma_\eps(A)\cap\overline{B_L(0)} = \bigcap_{\gamma>0}\bigcap_{s\in S_\gamma}\left[\left(B_{\delta_s}(W((A-s)^{-1}))\right)+s\right]\cap\overline{B_L(0)}.
\end{align*}
\end{enumerate}
\end{theo}
\begin{proof}
\begin{enumerate}
\item The first inclusion follows from Theorem \ref{theo:psincl-shift}. In order to prove the second inclusion first note that
\begin{align*}
S_\gamma\cap \bigcap_{{s}\in S_\gamma}
  \left[\bigl(B_{\delta_{{s}}}(W((A-{s})^{-1}))\bigr)^{-1}+{s}\right] = \varnothing
\end{align*}
for every $\gamma>0$ by Proposition~\ref{prop:rho}. Hence,
\begin{align*}
\bigcap_{\gamma>0}\bigcap_{s\in S_\gamma}\left[\left(B_{\delta_s}(W((A-s)^{-1}))\right)^{-1}+s\right] \subset \bigcap_{\gamma>0}\C\setminus S_\gamma = \C\setminus\bigcup_{\gamma>0} S_\gamma\\
=\C\setminus\set{s\in\varrho(A)}{\|(A-s)^{-1}\|<\frac{1}{\eps}} = \set{s\in\C}{\|(A-s)^{-1}\|\geq\frac{1}{\eps}}.
\end{align*}
\item First note that,
  analogously to Proposition~\ref{prop:psincl},
\begin{align*}
\set{\lambda\in\C}{\|(A-\lambda)^{-1}\|\geq\frac{1}{\eps}} \subset \left(\overline{B_\delta(W(A^{-1}))}\right)^{-1} 
\end{align*}
for $0\in\varrho(A)$, $0<\eps<\frac{1}{\|A^{-1}\|}$ and
$\delta = \frac{\|A^{-1}\|^2\eps}{1-\|A^{-1}\|\eps}$.
The proof of this statement is completely analogous to the one of
Proposition \ref{prop:psincl}.
Proceeding as in the proof of Theorem \ref{theo:psincl-shift}, we obtain
\begin{align*}
\set{\lambda\in\C}{\|(A-\lambda)^{-1}\|\geq\frac{1}{\eps}} \subset \bigcap_{\gamma>0}\bigcap_{s\in S_\gamma}\left[\left(\overline{B_{\delta_s}(W((A-s)^{-1}))}\right)^{-1}+s\right].
\end{align*}
The other inclusion can be shown as in part (a) since we also have
\begin{align*}
S_\gamma\cap \bigcap_{{s}\in S_\gamma}
  \left[\bigl(\overline{B_{\delta_{{s}}}(W((A-{s})^{-1}))}\bigr)^{-1}+{s}\right] = \varnothing
\end{align*}
for every $\gamma>0$ as a consequence of Proposition~\ref{prop:rho}.
\item By (a) it suffices to show that $\lambda\in\varrho(A)\cap\overline{B_L(0)}$, $\|(A-\lambda)^{-1}\|=\frac{1}{\eps}$ implies $\lambda\notin \left(B_{\delta_s}(W((A-s)^{-1}))\right)^{-1}+s$ for some $\gamma>0$ and $s\in S_\gamma$. Let 
\begin{align*}
\eps_1 =\frac{1}{2}\min\set{\dist(\mu,\sigma(A)\setminus \{\mu\})}{\mu\in\sigma(A)\cap\overline{B_L(0)}}.
\end{align*}
Since $A$ has compact resolvent, the minimum exists and is positive. With
\begin{align*}
\eps_0 = \frac{1}{2}\min\set{\dist(\mu,\sigma(A)\setminus \{\mu\})}{\mu\in\sigma(A)\cap\overline{B_{L+3\eps_1}(0)}}
\end{align*}
we then have $0<\eps_0\leq\eps_1$. Let now $\eps<\eps_0$ and $\lambda\in \varrho(A)\cap\overline{B_L(0)}$ with $\|(A-\lambda)^{-1}\|=\frac{1}{\eps}$. Since $A$ is normal, we get $\dist(\lambda,\sigma(A))=\eps$ and hence there exists a $\mu\in\sigma(A)$ such that $|\lambda-\mu|=\eps$. In particular we have $\mu\in B_{L+\eps_1}(0)$. Choose now $\gamma\in(0,\eps_0-\eps)$, i.e. $\eps<\eps+\gamma<\eps_0$, and set
\begin{align*}
s = \mu+\frac{\eps+\gamma}{\eps}(\lambda-\mu).
\end{align*}
Then $s\in B_{\eps_0}(\mu)$ and
\begin{align*}
\dist(s,\sigma(A)) = |\mu-s|=\eps+\gamma.
\end{align*}
Indeed if $\mu'\in\sigma(A)\cap\overline{B_{L+3\eps_1}(0)}$ with $\mu\neq\mu'$, then $B_{\eps_0}(\mu)\cap B_{\eps_0}(\mu')=\varnothing$ and hence $|\mu'-s|>\eps_0$. If $\mu'\in \sigma(A)$ and $|\mu'|>L+3\eps_1$, then $\dist(\mu',B_{\eps_0}(\mu))>\eps_1$ since $B_{\eps_0}(\mu)\subset B_{L+\eps_1+\eps_0}(0)$ and thus $|\mu'-s|>\eps_1\geq\eps_0$. Due to $|\mu-s|<\eps_0$ we therefore obtain $\dist(s,\sigma(A))=|\mu-s|$ and because $A$ is normal we can conclude
\begin{align*}
\|(A-s)^{-1}\| = \frac{1}{\eps+\gamma},
\end{align*}
i.e. $s\in S_\gamma$. Since
\begin{align}\label{eq:gamma}
\frac{1}{\delta_s+\|(A-s)^{-1}\|} = \left(\frac{\|(A-s)^{-1}\|}{1-\|(A-s)^{-1}\|\eps}\right)^{-1} = \frac{1}{\|(A-s)^{-1}\|}-\eps = \gamma,
\end{align}
Proposition~\ref{prop:rho} implies
\begin{align*}
\overline{B_\gamma(s)}\cap\left[\left(B_{\delta_s}(W((A-s)^{-1}))\right)^{-1}+s\right]=\varnothing.
\end{align*}
By our choice of $s$ we have $\lambda\in \overline{B_\gamma(s)}$ and thus
\begin{equation*}
  \lambda\notin \left(B_{\delta_s}(W((A-s)^{-1}))\right)^{-1}+s.
  \qedhere
\end{equation*}
\end{enumerate}
\end{proof}
}
\cw{
\begin{remark}
  \begin{enumerate}
  \item
    The statements of part (a) and (b) of the previous theorem
    continue to hold under the weaker assumption
    \(\delta_s\geq\frac{\|(A-s)^{-1}\|^2\eps}{1-\|(A-s)^{-1}\|\eps},\)
    i.e., equality is not needed there.
  \item
    The cutoff with the large ball $\overline{B_L(0)}$ in part (c)
    is not needed in the matrix case
    (i.e.\ $\dim H<\infty$), or if the
    eigenvalues of $A$ satisfy a uniform gap condition.
    On the other hand, the equality in (c)
    will typically not hold for all $\eps>0$, i.e. the restriction $\eps<\eps_0$ is
    needed, even in the matrix case.
    This is illustrated with the next (counter-)example.
  \end{enumerate}
\end{remark}
}
\lv{
\begin{example}
Let the normal matrix $A$ be given by
\begin{align*}
A = \left(\begin{matrix}1&0\\0&-1\end{matrix}\right)
\end{align*}
and consider $\eps=1$. Then $\sigma_\eps(A)=B_1(1)\cup B_1(-1)$ and in particular $0\notin\sigma_\eps(A)$. We will show that $0\in\left(B_{\delta_s}(W((A-s)^{-1}))\right)^{-1}+s$ for all $s\in S_\gamma$, $\gamma>0$. Hence
\begin{align*}
\sigma_\eps(A)\subsetneqq \bigcap_{\gamma>0}\bigcap_{s\in S_\gamma}\left[\left(B_{\delta_s}(W((A-s)^{-1}))\right)^{-1}+s\right]
\end{align*}
in this case. First observe that for $s\in S_\gamma$, i.e. $\|(A-s)^{-1}\|=\frac{1}{\eps+\gamma}$, we have $\frac{1}{\delta_s+\|(A-s)^{-1}\|} = \gamma$, see \eqref{eq:gamma}. This implies
\begin{align*}
\delta_s = \frac{1}{\gamma}-\|(A-s)^{-1}\| = \frac{1}{\gamma}-\frac{1}{\eps+\gamma} = \frac{\eps}{\gamma(\eps+\gamma)} = \frac{1}{\gamma(1+\gamma)}
\end{align*}
since $\eps=1$. We also have
\begin{align*}
(A-s)^{-1} = \left(\begin{matrix}(1-s)^{-1}&0\\0&(-1-s)^{-1}\end{matrix}\right)
\end{align*}
and hence
\begin{align*}
W((A-s)^{-1}) = \set{r(1-s)^{-1}+(1-r)(-1-s)^{-1}}{r\in[0,1]}.
\end{align*}
Due to $A$ being normal, $S_\gamma$ is the boundary of the $(1+\gamma)$-neighborhood of $\{-1,1\}$. Thus by taking $s_0\in S_\gamma$ with $\Real s_0 = 0$ we have 
\begin{align*}
|s|^2\geq|s_0|^2 = (1+\gamma)^2-1^2 = \gamma^2+2\gamma
\end{align*}
and hence $|s|>\gamma$. From
\begin{align*}
|(\pm 1-s)^{-1}-(-s^{-1})| = \left|\frac{1}{\pm 1-s}+\frac{1}{s}\right| = \frac{1}{|s||\pm 1-s|}\leq \frac{1}{|s|(1+\gamma)}
\end{align*}
we get
\begin{align*}
\dist&(-s^{-1},W((A-s)^{-1}))\\
&\leq\min_{r\in[0,1]}r\left|(1-s)^{-1}-(-s^{-1})\right|+(1-r)\left|(-1-s)^{-1}-(-s)^{-1}\right|\\
&\leq \frac{1}{|s|(1+\gamma)}<\frac{1}{\gamma(1+\gamma)} = \delta_s.
\end{align*}
This shows that $-s^{-1}\in B_{\delta_s}(W((A-s)^{-1}))$ and therefore
\begin{align*}
0\in \left(B_{\delta_s}(W((A-s)^{-1}))\right)^{-1}+s.
\end{align*}
\end{example}
}

\lv{
\begin{remark}\label{rem:NumericalRange}
Note that under the assumption $\sigma(A)\subset\overline{W(A)}$ (which holds for example if $A$ has a compact resolvent) it is known (see e.g.~\cite{ETBook} for the matrix case) that the pseudospectrum can also be enclosed by an $\eps$-neighborhood of the numerical range, namely
\begin{align}\label{eq:NumericalRange}
\sigma_\eps(A)\subset B_\eps(W(A)).
\end{align}
Indeed for $\lambda\in\sigma_\eps(A)\setminus\sigma(A)$ we have $\|(A-\lambda)^{-1}\|>\frac{1}{\eps}$ and therefore
\begin{align*}
\|(A-\lambda)x\|<\eps \qquad \text{ for all } x\in \mdef(A), \|x\|=1.
\end{align*}
This implies 
\begin{align*}
|\langle Ax,x\rangle-\lambda| = |\langle(A-\lambda)x,x\rangle| \leq \|(A-\lambda)x\|<\eps
\end{align*}
for $x\in\mdef(A)$, $\|x\|=1$. See Section \ref{sec:numericalexamples} for a comparison of the enclosure \eqref{eq:NumericalRange} with our method \eqref{eq:psincl-shift}.
\end{remark}
}

\section{A strong approximation scheme}
\label{sec:strappr}

In this section
we consider finite-dimensional approximations $A_n$ to the full operator
$A$.
Our aim is to prove a version of Theorem~\ref{theo:psincl-shift}
which provides a pseudospectrum enclosure for the full operator $A$
in terms of numerical ranges of the approximating matrices $A_n$;
this will allow us to compute the enclosure by numerical methods.

We suppose that $0\in\varrho(A)$ and consider
a sequence of approximations $A_n$
of the operator $A$ of the following form:
\begin{enumerate}
\item
  $U_n\subset H$, $n\in\N$, are finite-dimensional subspaces of the Hilbert space $H$.
\item
  $P_n\in\mlin(H)$ are projections (not necessarily orthogonal)
  onto $U_n$, i.e. $\range (P_n)=U_n$, such that
  \begin{equation}\label{eq:strong-proj-conv}
    \lim_{n\to\infty}P_nx=x\qquad\text{for all}\qquad x\in H.
  \end{equation}
\item
  $A_n\in\mlin(U_n)$ are invertible such that
  \begin{equation}\label{eq:strong-approx}
    \lim_{n\to\infty}A_n^{-1}P_nx= A^{-1}x \qquad\text{for all}\qquad
    x\in H.
  \end{equation}
\end{enumerate}
In this case we say that the family $(P_n,A_n)_{n\in\N}$
\emph{approximates $A$ strongly}.
Note that \eqref{eq:strong-proj-conv} implies
that $\bigcup_{n\in\N}U_n$ is dense in $H$ and that
$\sup_{n\in\N}\|P_n\|<\infty$
by the uniform boundedness principle.

\begin{lemma}\label{lem:konsistenz}
  Let $U_n$, $P_n$ be such that \eqref{eq:strong-proj-conv} holds and
  let $A_n\in\mlin(U_n)$ be invertible. Then
  the following assertions are equivalent:
  \begin{enumerate}
  \item $\lim_{n\to\infty}A_n^{-1}P_n x= A^{-1}x$ for all $x\in H$,
    i.e., \eqref{eq:strong-approx} holds.
  \item $\sup_{n\in\N}\|A_n^{-1}\|_{\mlin(U_n)}<\infty$ and
    for all $x\in\mdef(A)$ there exist $x_n\in U_n$ such that
    \[\lim_{n\to\infty}x_n=x,\quad\lim_{n\to\infty}A_nx_n=Ax.\]
  \end{enumerate}
\end{lemma}
\begin{proof}
  $(a)\Rightarrow(b)$.
  The uniform boundedness principle yields
  \[\sup_{n\in\N}\|A_n^{-1}P_n\|_{\mlin(H)}<\infty.\]
  Since $\|A_n^{-1}u\|=\|A_n^{-1}P_nu\|\leq\|A_n^{-1}P_n\|_{\mlin(H)}\|u\|$
  for all $u\in U_n$, this shows the first part.
  For the second, let $x\in\mdef(A)$ and set $y=Ax$ and
  $x_n=A_n^{-1}P_ny$. Then $x_n\to A^{-1}y=x$ and
  $A_nx_n=P_ny\to y=Ax$ as $n\to\infty$.

  $(b)\Rightarrow(a)$.
  Let $y\in H$. Set $x=A^{-1}y$ and choose $x_n\in U_n$ according to $(b)$.
  Then
  \[A_n^{-1}P_ny=
  A_n^{-1}P_nAx=
  A_n^{-1}(P_nAx-A_nx_n)+x_n.\]
  Since both $P_nAx\to Ax$ and $A_nx_n\to Ax$ as $n\to\infty$ and
  $\|A_n^{-1}\|$ is uniformly bounded, we obtain $(a)$.
\end{proof}

The following lemma shows that if $A$ is approximated by $A_n$ strongly,
then $A-\lambda$ is approximated by $A_n-\lambda$ strongly too,
provided $\|(A_n-\lambda)^{-1}\|$ is uniformly bounded in $n$.
\begin{lemma}\label{lem:shifted-strong-approx}
  Suppose that $(P_n,A_n)_{n\in\N}$ approximates $A$ strongly.
  If $\lambda\in\varrho(A)$ is such that $\lambda\in\varrho(A_n)$ for all
  $n\in\N$ and
  $\sup_{n\in\N}\|(A_n-\lambda)^{-1}\|<\infty$, then
  \[\lim_{n\to\infty}(A_n-\lambda)^{-1}P_nx=(A-\lambda)^{-1}x
  \qquad\text{for all}\qquad x\in H.\]
\end{lemma}
\begin{proof}
  This follows immediately from Lemma~\ref{lem:konsistenz} since
  \[\lim_{n\to\infty}A_nx_n=Ax \quad\Longleftrightarrow\quad
  \lim_{n\to\infty}(A_n-\lambda)x_n=(A-\lambda)x\]
  whenever $\lim_{n\to\infty}x_n=x$.
\end{proof}

\begin{remark}\label{rem:conv-notions}
  \cw{In the literature there is a variety of notions describing the
  approximation of a linear operator.
  Two notions that are close to our definition of a strong approximation scheme
  are
  \emph{generalized strong resolvent convergence}, considered in
  \cite{boegli-phd,boegli18,weidmann00},
  and \emph{discrete-stable convergence}, see \cite{chatelin}.
  There are however subtle differences between these two notions and our setting:
  First, we do \emph{not} assume that $P_n(\mdef(A))\subset U_n$.
  Second, in Lemma~\ref{lem:konsistenz}(b) we do not have
  the convergence of $A_nP_nx$ to $Ax$,
  which would be the case for discrete-stable convergence.
  Up to these differences, the results of Lemmas~\ref{lem:konsistenz}
  and~\ref{lem:shifted-strong-approx} are well known in the literature,
  see \cite[Lemma~1.2.2, Theorem~1.2.9]{boegli-phd} and
  \cite[Lemma~3.16]{chatelin}.
  }
\end{remark}

We now prove 
a convergence result for the numerical range of the inverse operator
under strong approximations.
\begin{lemma}\label{lem:nrinv-approx}
  Suppose that $(P_n,A_n)_{n\in\N}$ approximates $A$ strongly. Then
  \begin{enumerate}
  \item for every $x\in H$, $\|x\|=1$ there exists a sequence $y_n\in U_n$,
    $\|y_n\|=1$ such that
    \[\lim_{n\to\infty}\iprod{A_n^{-1}y_n}{y_n}=\iprod{A^{-1}x}{x};\]
  \item for all $\delta>0$ there exists $n_0\in\N$ such that
    \[W(A^{-1})\subset B_\delta\left(W(A_n^{-1})\right),
    \qquad n\geq n_0.\]
  \end{enumerate}
\end{lemma}
\begin{proof}
  \begin{enumerate}
  \item
    We set $y_n=P_nx/\|P_nx\|$. Note that $y_n$ is well defined for almost
    all $n$ since $\|P_nx\|\to\|x\|=1$.
    We get $y_n\to x$ as $n\to\infty$ and
    \begin{align*}
      &|\iprod{A^{-1}x}{x}-\iprod{A_n^{-1}y_n}{y_n}|\\
      &\leq|\iprod{A^{-1}x-A_n^{-1}P_nx}{x}|+|\iprod{A_n^{-1}P_nx}{x-y_n}|
      +|\iprod{A_n^{-1}(P_nx-y_n)}{y_n}|\\
      &\leq\|A^{-1}x-A_n^{-1}P_nx\|+\|A_n^{-1}\|\|P_nx\|\|x-y_n\|
      +\|A_n^{-1}\|\|P_nx-y_n\|,
    \end{align*}
    which yields the assertion.
  \item
    Since $W(A^{-1})$ is bounded, it is precompact and hence
    there exist $z_1,\dots,z_m\in W(A^{-1})$
    such that
    \[W(A^{-1})\subset\bigcup_{j=1}^m B_{\delta/2}(z_j).\]
    For every $j$ we have $z_j=\iprod{A^{-1}x_j}{x_j}$ with some
    $x_j\in H$, $\|x_j\|=1$, and by (a) there exists
    $n_j\in\N$ such that for all $n\geq n_j$ there is a $y_j\in U_n$,
    $\|y_j\|=1$ such that
    \[|\iprod{A^{-1}x_j}{x_j}-\iprod{A_n^{-1}y_j}{y_j}|<\frac{\delta}{2}.\]
    Hence
    \[W(A^{-1})\subset
    \bigcup_{j=1}^mB_\delta\left(\iprod{A_n^{-1}y_j}{y_j}\right)
    \subset B_\delta\left(W(A_n^{-1})\right)\]
    for all $n\geq n_0=\max\{n_1,\dots,n_m\}$.\qedhere
  \end{enumerate}
\end{proof}

The previous lemma allows us easily to prove an approximation
version of the basic enclosure result Proposition~\ref{prop:psincl}.
\begin{prop}\label{prop:psincl-str}
  Suppose that $(P_n,A_n)_{n\in\N}$ approximates $A$ strongly.
  For \lv{$0<\eps < \frac{1}{\|A^{-1}\|}$ and $\delta>\frac{\|A^{-1}\|^2\eps}{1-\|A^{-1}\|\eps}$}
  there exists $n_0\in\N$ such that
  \[\sigma_\eps(A)\subset\left(B_\delta(W(A_n^{-1}))\right)^{-1}
  \quad\text{for all}\quad n\geq n_0.\]
\end{prop}
\begin{proof}
  By Proposition~\ref{prop:psincl} we have
  \[\sigma_\eps(A)\subset\left(B_{\delta'}(W(A^{-1}))\right)^{-1}\]
  where \lv{$\delta'=\frac{\|A^{-1}\|^2\eps}{1-\|A^{-1}\|\eps}$}. Since $\delta-\delta'>0$,
  Lemma~\ref{lem:nrinv-approx} yields a constant $n_0\in\N$ such that
  \[W(A^{-1})\subset B_{\delta-\delta'}\left(W(A_n^{-1})\right),
  \qquad n\geq n_0.\]
  Consequently
  $B_{\delta'}(W(A^{-1}))\subset B_\delta(W(A_n^{-1}))$
  for $n\geq n_0$ and the proof is complete.
\end{proof}

Combining the previous proposition with shifts of the operator,
we get our second main result.
It is analogous to Theorem~\ref{theo:psincl-shift},
but provides an enclosure of the pseudospectrum of the
infinite-dimensional operator in terms of numerical ranges
of the approximating matrices.
\begin{theo}\label{theo:psincl-str-shift}
  Suppose that $(P_n,A_n)_{n\in\N}$ approximates $A$ strongly.
  Let the shifts $s_1,\dots,s_m\in\varrho(A)$ be such that
  \[\sup_{n\in\N}\|(A_n-s_j)^{-1}\|<\infty \quad\text{for all}\quad
  j=1,\dots,m.\]
  Let \lv{$0<\eps < \frac{1}{\max_{j=1,\dots,m}\|(A-s_j)^{-1}\|}$ and
  $\delta_j>\frac{\|(A-s_j)^{-1}\|^2\eps}{1-\|(A-s_j)^{-1}\|\eps}$} for all $j$.
  Then there exists $n_0\in\N$ such that
  \[\sigma_\eps(A)\subset \bigcap_{j=1}^m
  \left[\left(B_{\delta_j}(W((A_n-s_j)^{-1}))\right)^{-1}+s_j\right]
  \quad\text{for all}\quad n\geq n_0.\]
\end{theo}
\begin{proof}
  In view of Lemma~\ref{lem:shifted-strong-approx},
  Proposition~\ref{prop:psincl-str} can be applied to every $A-s_j$.
  Hence there exists $n_j\in\N$ such that
  \[\sigma_\eps(A-s_j)\subset\left(B_{\delta_j}(W((A_n-s_j)^{-1}))\right)^{-1},
  \qquad n\geq n_j.\]
  Since $\sigma_\eps(A)=\sigma_\eps(A-s_j)+s_j$, the claim follows
  with $n_0=\max\{n_1,\dots,n_m\}$.
\end{proof}

\section{A uniform approximation scheme}\label{sec:unifap}

In this section we pose additional assumptions on the
approximations $A_n$ of the infinite-dimensional operator $A$,
that will allow us to estimate the starting index $n_0$ for which the
pseudospectrum enclosures from
Proposition~\ref{prop:psincl-str} and
Theorem~\ref{theo:psincl-str-shift}
hold on  bounded sets.

Throughout this section we assume that $A$ has a compact resolvent,
$0\in\varrho(A)$ and that $\D(A)\subset W\subset H$ where the
Hilbert space $W$ is continuously and
densely embedded into $H$.
The closed graph theorem then implies
$A^{-1}\in \mlin(H,W)$.
Further, we suppose that there is a sequence of
approximations of the operator $A$ in the following sense:
\begin{enumerate}
\item
  $U_n\subset H$,
  $n\in\N$, are finite-dimensional subspaces of $H$.
\item There exist projections $P_n\in \mlin(H)$  onto $U_n$, $n\in\N$,
  not necessarily orthogonal,
  with $\sup_{n\in \N} \|P_n\|<\infty$ and
  $\|(I-P_n)|_W\|_{\mlin(W,H)}\rightarrow 0$ as $n\rightarrow \infty$.
\item There exist invertible operators $A_n\in \mlin(U_n)$, $n\in\N$,
  such that
  $\|A^{-1}-A^{-1}_nP_n\|\rightarrow 0$ as $n\rightarrow \infty$.
\end{enumerate}
We say that $(P_n,A_n)_{n\in\N}$ \emph{approximates $A$ uniformly}.
For $\|(I-P_n)|_W\|_{\mlin(W,H)}$ we will write abbreviatory
$\|I-P_n\|_{\mlin(W,H)}$.
\begin{remark}\label{rem:unifap}
  \begin{enumerate}
  \item
    Property (c) already implies that $A$ has compact resolvent:
    indeed
    $A^{-1}$  is the uniform limit of the finite rank operators
    $A_n^{-1}P_n$ and hence compact.
  \item
    If $(P_n,A_n)_{n\in\N}$ approximates $A$ uniformly, then also strongly.
    Note here that from (b) we first obtain $P_nx\to x$ for $x\in W$, which can
    then be extended to all $x\in H$ by the  density of $W$ in $H$ and the uniform
    boundedness of the $P_n$.
    One particular consequence of the strong approximation is
    \[\sup_{n\in \N} \|A^{-1}_n\|<\infty,\]
    see Lemma~\ref{lem:konsistenz}.
  \item
    \cw{%
    Property (c) amounts to the convergence of $A_n$ to $A$ in
    \emph{generalized norm resolvent sense}, see
    \cite{boegli-phd,boegli18,weidmann00} for this notion.
    Note however that our setting has the additional assumption
    that $P_n\to I$ \emph{uniformly} in $L(W,H)$
    where $\mdef(A)\subset W\subset H$.
    For generalized norm resolvent convergence this is not the case,
    but it will be a crucial element in the following proofs.
    }
  \end{enumerate}
\end{remark}

In order to obtain improved enclosures of the pseudospectrum 
under a uniform approximation scheme, that is, 
additional  estimates of the starting index $n_0$ for which the
pseudospectrum enclosures from
Proposition~\ref{prop:psincl-str} and
Theorem~\ref{theo:psincl-str-shift}
hold on  bounded sets,
we  refine
the results from Section~\ref{sec:psincl}
in terms of certain subsets of the full  numerical range of $A^{-1}$.
For $d>0$ we define
\begin{equation}\label{eq:nrd}
  W(A^{-1},d)= \set{\iprod{A^{-1}x}{x}}{\|x\|=1,\,x\in W,\, \|x\|_W\le d}.
\end{equation}
Clearly $W(A^{-1},d)\subset W(A^{-1})$. Moreover since $W$ is dense in $H$
we get
\begin{equation}\label{eq:nrd-limit}
  \overline{\bigcup_{d>0}W(A^{-1},d)}=\overline{W(A^{-1})}.
\end{equation}

\begin{prop}\label{prop:psincl2}
  Let $L>0$ and $d=L\|A^{-1}\|_{\mlin(H,W)}$. Then
  \begin{enumerate}
  \item $\sigma(A)\cap \overline{B_L(0)}\subset W(A^{-1},d)^{-1}$.
  \item
    If in addition \lv{$0<\eps < \frac{1}{\|A^{-1}\|}$},
    $L>\eps$ and \lv{$\delta=\frac{\|A^{-1}\|^2\eps}{1-\|A^{-1}\|\eps}$} then
    \[ \sigma_\varepsilon(A)\cap\overline{B_{L-\varepsilon}(0)} \subset
    \left(B_\delta(W(A^{-1},d))\right)^{-1}.\]
  \end{enumerate}
\end{prop}

\begin{proof}
  \begin{enumerate}
  \item Let $\lambda\in \sigma(A)$ with $|\lambda|\le L$. Then there exists $x\in \D(A)$ with $\|x\|=1$ and $Ax=\lambda x$. This implies 
    \[ \frac{1}{|\lambda|}\|x\|_W= \|A^{-1}x\|_W\le \|A^{-1}\|_{\mlin(H,W)} \|x\|=\|A^{-1}\|_{\mlin(H,W)} \]
    and thus we obtain
    \[ \|x\|_W \le \|A^{-1}\|_{\mlin(H,W)}|\lambda|\le L \|A^{-1}\|_{\mlin(H,W)}=d.\]
    Consequently $\lambda^{-1}= \iprod{A^{-1}x}{x}\in W(A^{-1},d)$.
  \item
    The proof is similar to the one of Proposition~\ref{prop:psincl}.
    We set $U=\left(B_\delta(W(A^{-1},d))\right)^{-1}$ and first show
    \begin{equation}\label{eq:regulartype2}
      \|(A-\lambda)x\|\geq\eps \quad\text{for all}\quad
      \lambda\in\overline{B_{L-\eps}(0)}\setminus U,\,x\in\mdef(A),\,\|x\|=1.
    \end{equation}
    Let $\lambda\in\overline{B_{L-\eps}(0)}\setminus U$, $x\in\mdef(A)$, $\|x\|=1$.
    We consider three cases.
    Suppose first that \lv{$|\lambda|>\frac{1}{\delta+\|A^{-1}\|}$} and $\|x\|_W\leq d$.
    From $\lambda\not\in U$ we obtain
    $\dist(\lambda^{-1},W(A^{-1},d))\geq\delta$,
    which implies
    \[\delta\leq|\lambda^{-1}-\iprod{A^{-1}x}{x}|
    =|\iprod{(\lambda^{-1}-A^{-1})x}{x}|
    \leq\|(\lambda^{-1}-A^{-1})x\|\]
    and thus
    \[\|(A-\lambda)x\|
    \geq\frac{|\lambda|}{\|A^{-1}\|}\|(\lambda^{-1}-A^{-1})x\|
    \geq\lv{\frac{\delta}{\|A^{-1}\|(\delta+\|A^{-1}\|)}}=\eps.\]
    In the second case assume $\|x\|_W\geq d$. Then
    \[ d\le \|x\|_W \le \|A^{-1}\|_{\mlin(H,W)} \|Ax\|,\]
    which in view of $\lambda\in\overline{B_{L-\eps}(0)}$ implies
    \[  \|(A-\lambda)x\| \ge \|Ax\| -|\lambda|
    \ge \frac{d}{\|A^{-1}\|_{\mlin(H,W)} }-|\lambda|=L-|\lambda|
    \geq\eps.\]
    Finally if \lv{$|\lambda|\leq\frac{1}{\delta +\|A^{-1}\|}$}, the same reasoning as in
    the proof of Proposition~\ref{prop:psincl} yields
    once again that $\|(A-\lambda)x\|\geq\eps$,
    and therefore \eqref{eq:regulartype2} is proved.
    Now, since $A$ has a compact resolvent \eqref{eq:regulartype2} implies that
    \[\lambda\in\overline{B_{L-\eps}(0)}\setminus U
    \quad\Rightarrow\quad \lambda\in\varrho(A),\,
    \|(A-\lambda)^{-1}\|\leq\frac{1}{\eps}.\]
    Consequently
    $\sigma_\eps(A)\cap\overline{B_{L-\eps}(0)}\subset U$.\qedhere
  \end{enumerate}
\end{proof}

From Proposition~\ref{prop:psincl2} we get again a shifted version:
\begin{theo}\label{theo:psincl2-shift}
  Let $S\subset\varrho(A)$ be such that
  \[M_0:=\sup_{s\in S}\|(A-s)^{-1}\|<\infty,\qquad
  M_1:=\sup_{s\in S}\|(A-s)^{-1}\|_{\mlin(H,W)}<\infty.\]
  For \lv{$0<\eps < \frac{1}{M_0}$}, $L>\eps$,
  $d=LM_1$ and \lv{$\delta_s=\frac{\|(A-s)^{-1}\|^2\eps}{1-\|(A-s)^{-1}\|\eps}$}
  we get the inclusion
  \[\sigma_\varepsilon(A)\cap\bigcap_{s\in S}\overline{B_{L-\varepsilon}(s)} \subset
  \bigcap_{s\in S}\left[\left(B_{\delta_s}(W((A-s)^{-1},d))\right)^{-1}+s\right].\]
\end{theo}
\begin{proof} Apply Proposition~\ref{prop:psincl2}(b) to $A-s$ for all $s\in S$
  and note that
  \[\lambda\in\sigma_\eps(A-s)\cap\overline{ B_{L-\eps}(0)}
  \quad\Leftrightarrow\quad
  \lambda+s\in\sigma_\eps(A)\cap\overline{B_{L-\eps}(s)}.\qedhere\]
\end{proof}

\begin{remark}
  By the continuity of the embedding $W\hookrightarrow H$,
  the condition $M_1<\infty$
  already implies $M_0<\infty$.
\end{remark}

For a uniform approximation scheme, the numerical range
of $A^{-1}$ can now be approximated with explicit control on the
starting index $n_0$:

\begin{lemma}\label{lem:nrinv-unifapprox}
  Suppose that $(P_n,A_n)_{n\in\N}$ approximates $A$ uniformly.
  Let
  \begin{equation}\label{eq:C0}
    C_0=\sup_{n\in \N}\left( \|A^{-1}_n\| \|P_n\|+ 6 \|A^{-1}_n\|\|P_n\|^2\right).
  \end{equation}
  \begin{enumerate}
  \item If $d>0$, $0<\delta\leq\frac{C_0}{2}$ and $n_0\in\N$ are such that for
    every $n\ge n_0$
    \[\|A^{-1}- A^{-1}_n P_n\|  + d C_0   \| I-P_n\|_{\mlin(W,H)}<\delta,\]
    then
    \begin{align*}
      W(A^{-1},d)\subset B_\delta(W(A^{-1}_n)), \qquad n\geq n_0.
    \end{align*}
  \item If $\delta>0$ and $n_0\in\N$ are such that for every $n\ge n_0$ we have
    $\|A^{-1}- A^{-1}_n P_n\|<\delta$,
    then
    \begin{align*}
      W(A^{-1}_n)\subset B_\delta(W(A^{-1})),\qquad n\geq n_0.
    \end{align*}
  \end{enumerate}
\end{lemma}
\begin{proof}
  Let $x\in W$ with $\|x\|=1$ and  $\|x\|_W\le d$. Then we obtain
  \begin{align*}
    &|\iprod{A^{-1}x}{x}- \iprod{A^{-1}_n P_nx}{P_nx}|\\
    &\le |\iprod{A^{-1}x- A^{-1}_n P_nx}{x}| + |\iprod{A^{-1}_n P_nx}{x-P_nx}|\\
    &\le \|A^{-1}- A^{-1}_n P_n\| \|x\|^2 + \|A^{-1}_n\| \|P_n\| \|x\| \| I-P_n\|_{\mlin(W,H)} \| x \|_W\\
    &\le  \|A^{-1}- A^{-1}_n P_n\|  + d \|A^{-1}_n\| \|P_n\| \| I-P_n\|_{\mlin(W,H)}.
  \end{align*}
  as well as
  \begin{align*}
    |1-\|P_nx\|| \le \| x-P_n x\|\le  \| I-P_n\|_{\mlin(W,H)}\|x\|_W
    \leq d\|I-P_n\|_{\mlin(W,H)}.
  \end{align*}
  Let $n\ge n_0$. Then
  \[|1-\|P_nx\||\leq d\|I-P_n\|_{\mlin(W,H)}
  <\frac{\delta}{C_0}\leq\frac12\]
  and hence
  $\|P_nx\|\ge \frac{1}{2}$. Let $x_n= \frac{P_nx}{\|P_nx\|}$.
  Then  $\|x_n\|=1$ and 
  \begin{align*}
    \left|1-\frac{1}{\|P_nx\|^2}\right| & = \left|\frac{\|P_nx\|^2-1}{\|P_nx\|^2}\right| \\
    & = \frac{(\|P_nx\|+1)|\|P_nx\|-1|}{\|P_nx\|^2}\\
    &= \left( \frac{1}{\|P_nx\|}+\frac{1}{\|P_nx\|^2}\right) |1-\|P_nx\||\\
    &\le 6 |1-\|P_nx\||\\
    &\le 6d \| I-P_n\|_{\mlin(W,H)}.
  \end{align*}
  This implies
  \begin{align*}
    |\iprod{A^{-1}_n P_nx}{P_nx}&- \iprod{A^{-1}_nx_n}{x_n}|\\
    &=\left|\iprod{A^{-1}_n P_nx}{P_nx}- \frac{\iprod{A^{-1}_nP_n x}{P_n x}}{\|P_nx\|^2}\right|\\
    &= \left|1 - \frac{1}{\|P_nx\|^2}\right||\iprod{A^{-1}_n P_nx}{P_nx}|\\
    &\le 6 d \| I-P_n\|_{\mlin(W,H)}  \|A^{-1}_n\| \|P_n\|^2,
  \end{align*}
  and thus for $n\ge n_0$ we arrive at
  \begin{align*}
    &|\iprod{A^{-1}x}{x}-\iprod{A^{-1}_n x_n}{x_n}|\\
    &\leq\|A^{-1}- A^{-1}_n P_n\|
    + d \| I-P_n\|_{\mlin(W,H)}( \|A^{-1}_n\| \|P_n\|+6\|A_n^{-1}\|\|P_n\|^2)\\
    &\le \|A^{-1}- A^{-1}_n P_n\|  + d C_0   \| I-P_n\|_{\mlin(W,H)}\\
    &<\delta.
  \end{align*}
  This yields $\iprod{A^{-1}x}{x}\in B_\delta(W(A^{-1} _n))$ if $n\ge n_0$ and  proves (a). 

  In order to show part (b), let $x\in U_n$ with $\|x\|=1$. As $x=P_nx$ we have
  \begin{align*}
    |\iprod{A^{-1}_nx}{x}- \iprod{A^{-1}x}{x}|
    &\le \|A^{-1}_n x- A^{-1}x\| \|x\| \\
    &= \|A^{-1}_nP_n x- A^{-1}x\|  \le \|A^{-1}- A^{-1}_n P_n\|.
  \end{align*}
  Thus $\iprod{A^{-1}_nx}{x}\in B_\delta(W(A^{-1}))$ for $n\ge n_0$. 
\end{proof}

\begin{coroll}\label{cor1}
  If  $(P_n,A_n)_{n\in\N}$ approximates $A$ uniformly, then
  \begin{align*}
    \overline{W(A^{-1})}=\{ \lambda\in\C\mid \exists (\lambda_n)_{n\in\N}
    \text{ with } \lambda_n\in W(A_n^{-1}) \text{ and }\lim_{n\rightarrow
      \infty} \lambda_n= \lambda\}
  \end{align*}
  or, equivalently,
  \[\displaystyle\overline{W(A^{-1})}=\bigcap_{m\in\N}
  \overline{\bigcup_{n\ge m} W(A_n^{-1})}.\]
\end{coroll}
\begin{proof}
  We first show the inclusion "$\supset$". Let $(\lambda_n)_{n\in\N}$ be
  a convergent sequence in $\C$ with $\lambda_n\in W(A_n^{-1})$ and
  define $\lambda=\lim_{n\rightarrow \infty} \lambda_n$. Let
  $\delta>0$ be arbitrary. Lemma~\ref{lem:nrinv-unifapprox}(b) implies that there
  exists $n_0\in\N$ such that $\lambda_n\in B_\delta(W(A^{-1}))$
  for every $n\ge n_0$. This implies $\lambda \in
  B_\delta(W(A^{-1}))$ for every $\delta>0$, and thus $\lambda
  \in \overline{W(A^{-1})}$.

  Conversely, let $\lambda\in W(A^{-1},d)$ for some $d>0$.
  Using Lemma~\ref{lem:nrinv-unifapprox}(a), we can construct a sequence
  $(\lambda_n)_{n\in\N}$ in $\C$ with $\lambda_n\in W(A_n^{-1})$ and
  $\lambda=\lim_{n\rightarrow \infty} \lambda_n$. The statement now
  follows from \eqref{eq:nrd-limit}.
\end{proof}

The last result shows that $\overline{W(A^{-1})}$ can be represented as the
pointwise limit of the finite-dimensional numerical ranges $W(A_n^{-1})$.
Lemma~\ref{lem:nrinv-unifapprox} even yields a  uniform approximation,
but this is asymmetric, since one inclusion only holds for the restricted
numerical range $W(A^{-1},d)$.
A more symmetric result is discussed in the next remark:

\begin{remark}\label{rem1}
  If $U_n\subset W$ for some $n\in\N$ then, due to the fact that
  the space $U_n$ is finite-dimensional,
  \begin{align*}
    d_n:=\sup_{x\in U_n}\frac{\|x\|_W}{\|x\|}<\infty.
  \end{align*}
  Using the same reasoning as in the proof of Lemma~\ref{lem:nrinv-unifapprox}(b),
  we then obtain
  \begin{align*}
    W(A^{-1}_n)\subset B_\delta (W(A^{-1}, d_n))
  \end{align*}
  if $\|A^{-1}- A^{-1}_n P_n\|<\delta$.
  \end{remark}
  
  Note however that for  finite element discretization schemes
  the condition $U_n\subset W$ will usually \emph{not} be fulfilled.
  In our examples for instance $U_n$ are
  piecewise linear finite elements
  while $W\subset H^2(\Omega)$ is a second order Sobolev space,
  and thus $U_n\not\subset W$.

Under a uniform approximation scheme the pseudospectrum can be
approximated as follows.

\begin{prop}\label{prop:psincl-unif}
  Suppose that $(P_n,A_n)_{n\in\N}$ approximates $A$ uniformly.
  Let
  \[r>0,\quad \lv{0<\varepsilon < \frac{1}{\|A^{-1}\|} \quad\text{and}\quad
  \frac{\|A^{-1}\|^2\varepsilon}{1-\|A^{-1}\|\eps}<\delta\leq \frac{\|A^{-1}\|^2\varepsilon}{1-\|A^{-1}\|\eps} + \frac{7}{2}\|A^{-1}\|} .\]
  If we choose
  $n_0\in\N$ such that for every $n\ge n_0$
  \[ \|A^{-1}- A^{-1}_n P_n\|
  + (r+\varepsilon) \|A^{-1}\|_{\mlin(H,W)} C_0   \| I-P_n\|_{\mlin(W,H)}
  \lv{<\delta- \frac{\|A^{-1}\|^2\varepsilon}{1-\|A^{-1}\|\eps}}, \]
  where $C_0$ is defined in \eqref{eq:C0},
  then we obtain
  \[  \sigma_\varepsilon(A)\cap\overline{B_{r}(0)}
  \subset \left(B_\delta(W(A^{-1}_n))\right)^{-1}
  \quad\text{for all}\quad n\ge n_0. \]
\end{prop}

\begin{proof}
  Let \lv{$\delta'= \frac{\|A^{-1}\|^2\varepsilon}{1-\|A^{-1}\|\eps}$}, $L=r+\varepsilon$ and
  $d=L \|A^{-1}\|_{\mlin(H,W)}$.
  Proposition~\ref{prop:psincl2} implies
  \[ \sigma_\varepsilon(A)\cap\overline{B_{r}(0)}
  \subset  (B_{\delta'}(W(A^{-1},d)))^{-1}.\]
  Next note that
  \lv{
  \begin{align*}
  \delta-\delta'&\leq\frac{7}{2}\|A^{-1}\|=\lim_{n\to\infty}\frac{7}{2}\|A_n^{-1}P_n\|
  \leq\frac{1}{2}\limsup_{n\to\infty}\left(\|A_n^{-1}\|\|P_n\|+6\|A_n^{-1}\|\|P_n\|^2\right)\\
  &\leq\frac{C_0}{2},
  \end{align*}
  because $P_n$ is a projection.}
  We can therefore apply Lemma~\ref{lem:nrinv-unifapprox} 
  with $\delta$ replaced by $\delta-\delta'$ and $n_0$ chosen as stated above
  and obtain
  \[W(A^{-1},d)\subset B_{\delta-\delta'}(W(A_n^{-1})) \quad\text{for}\quad
  n\geq n_0\]
  and hence the assertion.
\end{proof}

\section{Finite element discretization of elliptic partial differential operators}
\label{sec:fe-discr}

As an example for a uniform approximation scheme defined in
Section~\ref{sec:unifap} we now consider finite element discretizations.
We use the  standard textbook approach via form methods,
which can be found  e.g.\ in \cite{Arendt,Sho}.

Let $V$ and $H$ be Hilbert spaces with $V\subset H$ densely and
continuously embedded.  In particular there is a constant $c>0$ such
that
\begin{equation}\label{eq:emb-const}
  \|x\|\leq c\|x\|_V,\qquad x\in V.
\end{equation}
Moreover, we consider a bounded and coercive sesquilinear form
$a:V\times V\rightarrow \C$, that is, there exists constants $M,
\gamma >0$ such that
\begin{equation}\label{eq:coercive}
  |a(x,y)|\le M \|x\|_V \|y\|_V \quad \mbox{and}\quad
  \Real a(x,x)\ge \gamma \|x\|^2_V, \quad x,y\in V.
\end{equation}
Let $A:\D(A)\subset H\rightarrow H$ be the operator associated with $a$,
which is given by
\begin{align*}
  \D(A) &= \{ x\in V\mid \exists c_x>0: |a(x,y)|\le c_x\|y\|\mbox{ for
  }y\in V\},\\ a(x,y)&= \iprod{Ax}{y}, \quad x\in D(A),y\in V.
\end{align*}
Then $A$ is a densely defined, closed operator
with $0\in \varrho(A)$ and $\|A^{-1}\|\le \frac{c^2}{\gamma}$, where
$c>0$ is the constant from \eqref{eq:emb-const}.

Let $(U_n)_{n\in\N}$
be a sequence of finite-dimensional  subspaces of $V$ which are nested,
that is $U_n\subset U_{n+1}$.
We denote by $a_n=a|_{U_n}$ the restriction of $a$ from
$V$ to $U_n$. The form $a_n$ is again bounded and coercive with the
same constants $M$ and $\gamma$. Let $A_n\in\mlin(U_n)$ be the operator
associated with $a_n$, i.e.
\[ a(x,y) =\iprod{A_n x}{y}, \qquad x, y \in U_n.\]
Then again $0\in \varrho(A_n)$ and $\|A^{-1}_n\|\le
\frac{c^2}{\gamma}$.
Let $P_n\in\mlin(H)$ be the orthogonal
projection onto $U_n$. Thus $\|P_n\|=1$ and $A_n=P_n A_{n+1}|_{U_n}$,
that is, $A_n$ is a compression of $A_{n+1}$.

To obtain a uniform approximation scheme, we now consider an
additional Hilbert space $W$ which is densely and
continuously embedded into $H$ such that $\D(A)\subset W\subset V$.
We assume that there exists a sequence of operators $Q_n\in\mlin(W,V)$ with
$\range(Q_n)\subset U_n$ and
\begin{equation}\label{eq:interpol-op}
  \lim_{n\to\infty}\|I-Q_n\|_{\mlin(W,V)}=0.
\end{equation}
  
\begin{lemma}\label{lem:fe-unifapprox}
  For all $n\in\N$ the estimates
  \begin{align*}
    \| I-P_n\|_{\mlin(W,H)}&\le c \|I-Q_n\|_{\mlin(W,V)}, \\
    \|A^{-1}- A^{-1}_n P_n\| &\le \frac{cM}{\gamma}
    \|A^{-1}\|_{\mlin(H,W)} \|I-Q_n\|_{\mlin(W,V)}
  \end{align*}
  hold.
  In particular, the family $(P_n,A_n)_{n\in\N}$ approximates $A$ uniformly.
\end{lemma}
\begin{proof}
  For $w\in W$ we calculate
  \begin{align*}
    \|w-P_nw\| &= \inf_{u\in U_n}\|w-u\| \le \|w-Q_nw\|\le
    c\|w-Q_nw\|_V\\ & \le c \|I-Q_n\|_{\mlin(W,V)} \|w\|_W,
  \end{align*}
  which shows the first assertion. Moreover, for $f\in H$ we set $x=A^{-1}f$
  and $x_n=A_n^{-1}P_nf$. Then we obtain
  \begin{align*}
    a(x,y)&= \iprod{Ax}{y} = \iprod{f}{y},\quad y\in V,\\
    a(x_n,u)&=\iprod{A_nx_n}{u} = \iprod{P_nf}{u}= \iprod{f}{u},\quad u\in U_n.
  \end{align*}
  Using the Lemma of Cea \cite[Theorem VII.5.A]{Sho},
  we find
  \begin{align*}
    \|A^{-1}f-A_n^{-1}P_nf\| &=\|x-x_n\|\le c\|x-x_n\|_V\le \frac{cM}{\gamma}
    \inf_{u\in U_n}\|x-u\|_V\\ &\le \frac{cM}{\gamma} \|x-Q_n x\|_V \le \frac{cM}{\gamma}
    \|I-Q_n\|_{\mlin(W,V)}\|x\|_W \\ & \le \frac{cM}{\gamma} \|I-Q_n\|_{\mlin(W,V)}
    \|A^{-1}\|_{\mlin(H,W)}\|f\|,
  \end{align*}
  which implies the second assertion.
\end{proof}

\begin{theo}\label{theo:psincl-fe}
  Let $A$ be the operator associated with the coercive form $a$
  and let $A_n$, $Q_n$ be as above. 
  Let
  \[r>0,\quad \lv{0<\varepsilon < \frac{1}{\|A^{-1}\|} \quad\text{and}\quad
  \frac{\|A^{-1}\|^2\varepsilon}{1-\|A^{-1}\|\eps}<\delta\leq \frac{\|A^{-1}\|^2\varepsilon}{1-\|A^{-1}\|\eps} + \frac{7}{2}\|A^{-1}\|}.\]
  If
  $n_0\in\N$ is such that for every $n\ge n_0$
  \lv{\[ \|I-Q_n\|_{\mlin(W,V)}<
  \frac{\delta-\frac{\|A^{-1}\|^2\varepsilon}{1-\|A^{-1}\|\eps}}
       {c\|A^{-1}\|_{\mlin(H,W)}\left(\frac{M}{\gamma}
         + (r+\varepsilon)\frac{7c^2}{\gamma}\right)},\]}
  then
  \[\sigma_\varepsilon(A)\cap\overline{B_{r}(0)}
  \subset \left(B_\delta(W(A^{-1}_n))\right)^{-1}
  \quad\text{for all}\quad n\ge n_0. \]
\end{theo}
\begin{proof}
  We check that the conditions of Proposition~\ref{prop:psincl-unif}
  are satisfied:
  Using Lemma~\ref{lem:fe-unifapprox}, we estimate for $n\geq n_0$ and
  with $C_0$ from \eqref{eq:C0},
  \begin{align*}
    &\|A^{-1}-A_n^{-1}P_n\|+(r+\eps)\|A^{-1}\|_{\mlin(H,W)}C_0\|I-P_n\|_{\mlin(W,H)}\\
    &\leq c \|A^{-1}\|_{\mlin(H,W)}\|I-Q_n\|_{\mlin(W,V)}
    \left( \frac{M}{\gamma}  + (r+\varepsilon) \frac{7c^2}{\gamma} \right)
    \lv{<\delta-\frac{\|A^{-1}\|^2\varepsilon}{1-\|A^{-1}\|\eps}}.\qedhere
  \end{align*}
\end{proof}

\begin{example}\label{ex:elliptic-op}
  Let $\Omega\subset \R^ 2$ be a bounded, open, convex domain with
polygonal boundary $\Gamma$ and $\Gamma_D\subset \Gamma$ a union of
polygons of $\Gamma$. Let
\[V=H_0^1(\Omega),\]
equipped with the $H^1$-norm.
On $V$ we consider the sesquilinear form
\begin{equation}\label{form}
  a(u,v)=\int_\Omega \left(\sum_{i,j=1}^2 a_{ij} u_{x_i}\overline{v}_{x_j}+ \sum_{i=1}^2 b_{i} u_{x_i}\overline{v}+cu\overline{v}\right)dx,
\end{equation}
where $a_{ij}\in C^{0,1}(\overline\Omega)$ and $b_i,c\in L^\infty(\Omega)$.
We suppose that $a$ is coercive and uniformly elliptic.
Let $\{\mathcal{T}_n\}_{n\in\N}$ be a family of
nested, admissible and quasi-uniform triangulations of $\Omega$
satisfying $\sup_{T\in \mathcal{T}_n} {\rm diam}(T)\le \frac{1}{n}$. Let
\[ W= H^2(\Omega)\cap H_0^1(\Omega),\]
equipped with the $H^2$-norm, and
\[ U_n=\set{ u\in C^0(\overline{\Omega})}{ u|_T\in \bbP_1(T),
  T\in \mathcal{T}_n, u|_{\Gamma}=0},\quad n\in\N.\]
Here $\bbP_1(T)$ denotes the set of polynomials of degree 1 on the triangle $T$.
We get $U_n\subset V$. 
Moreover, the operator $A$ associated with $a$ is given by
\begin{align*}
  Au&=-\sum_{i,j=1}^2 \partial_{x_j}(a_{ij} u_{x_i})
  + \sum_{i=1}^2 b_{i} u_{x_i}+cu,\\
  \mdef(A)&=W.
\end{align*}
For the proof of $\mdef(A)=W$ we refer to
\cite[Theorem~3.2.1.2 and \S2.4.2]{Gri85}.

By the Sobolev embedding theorem we have
$H^2(\Omega)\hookrightarrow C^0(\overline{\Omega})$.
For $u\in W$ we define $Q_nu$ as the unique
element of $U_n$ satisfying $(Q_nu)(x)=u(x)$ for every vertex of the
triangulation $\mathcal{T}_n$.
Then $Q_n\in\mlin(W,V)$
with $\range(Q_n)\subset U_n$.
Moreover, \cite[Theorem 9.27]{Arendt} implies
that there is a constant $K>0$ such that
\[ \|I-Q_n\|_{\mlin(W,V)}\le \frac{K}{n},\qquad n\in\N.\]
We conclude that Theorem~\ref{theo:psincl-fe} can be applied in this example
with $n_0\in\N$ chosen such that
\[n_0>\frac{Kc\|A^{-1}\|_{\mlin(H,W)}
  \left(\frac{M}{\gamma}  + (r+\varepsilon)\frac{7c^2}{\gamma}\right)}
   {\delta-2\|A^{-1}\|^2\varepsilon}.
\]
\end{example}

Note that in Example~\ref{ex:elliptic-op} we can also consider $\Omega$
to be an open interval in $\R$.
All results continue to hold in an analogous way.

  

  

\section{Discretization of a structured block operator matrix}
\label{sec:discr-block}

In this section we investigate discretizations of a certain kind
of block operator matrices.
We consider block matrices of the form
\[\A=\pmat{A&B\\B^*&D}\]
where $A$ is a closed, densely defined operator $A:\mdef(A)\subset H\to H$
on the Hilbert space $H$, and $B,D\in\mlin(H)$.
Then the block matrix $\A$ is a closed, densely defined operator on the
product space $H\times H$ with domain $\mdef(\A)=\mdef(A)\times H$.
Additionally we assume that $0\in\varrho(A)$, $0\in\varrho(D)$ and
that both $A$ and $-D$ are \emph{uniformly accretive}, i.e., there exist
constants $\gamma_A,\gamma_D>0$ such that
\begin{align}
  \label{eq:A-uaccret}
  \Real\iprod{Ax}{x}&\geq\gamma_A\|x\|^2, \qquad x\in\mdef(A),\\
  \label{eq:D-uaccret}
  \Real\iprod{Dx}{x}&\leq-\gamma_D\|x\|^2, \qquad x\in H.
\end{align}

In the next lemma we show that under the above assumptions
there is  a gap in the spectrum of $\A$ along the imaginary axis,
and we also prove an estimate for the norm of the resolvent.
Similar results
were obtained in \cite{lan-tret98,lan-tret01} under
the additional assumption that
$A$ is sectorial and,  in \cite{lan-tret01},
without the condition that $B$ and $D$ are bounded.
However, no corresponding resolvent estimates
were shown.
We remark that the boundedness of $D$ is  not essential
in Lemma~\ref{lem:bomgap} but will be used thereafter.
\begin{lemma}
  \label{lem:bomgap}
  We have
  \(\set{\lambda\in\C}{-\gamma_D<\Real\lambda<\gamma_A}\subset\varrho(\A)\)
  and
  \[\|(\A-\lambda)^{-1}\|\leq\frac{1}{\min\{\gamma_A-\Real\lambda,
    \gamma_D+\Real\lambda\}},\quad -\gamma_D<\Real\lambda<\gamma_A.\]
\end{lemma}
\begin{proof}
  Consider the block operator matrix
  \[J=\pmat{I&0\\0&-I}.\]
  A simple calculation shows that for
  \(\lambda\in U:=\set{\lambda\in\C}{-\gamma_D<\Real\lambda<\gamma_A}\)
  and $x\in\mdef(A)$, $y\in H$,
  \begin{align*}
    &\Real\bigiprod{J(\A-\lambda)\pmat{x\\y}}{\pmat{x\\y}}\\
    &=\Real\bigl(\iprod{(A-\lambda)x}{x}+\iprod{By}{x}-\iprod{B^*x}{y}
    -\iprod{(D-\lambda)y}{y}\bigr)\\
    &=\Real\iprod{(A-\lambda)x}{x}-\Real\iprod{(D-\lambda)y}{y}\\
    &\geq(\gamma_A-\Real\lambda)\|x\|^2+(\gamma_D+\Real\lambda)\|y\|^2\\
    &\geq c_\lambda\left\|\pmat{x\\y}\right\|^2,
  \end{align*}
  where $c_\lambda=\min\{\gamma_A-\Real\lambda,\gamma_D+\Real\lambda\}$.
  It follows that
  \[\|J(\A-\lambda)v\| \|v\|
  \geq|\iprod{J(\A-\lambda)v}{v}|
  \geq c_\lambda\|v\|^2, \qquad v\in\mdef(\A),\]
  and therefore, since $\|Jw\|=\|w\|$ for all $w\in H\times H$,
  \begin{equation}\label{eq:bom-reg}
    \|(\A-\lambda)v\|\geq c_\lambda\|v\|,\qquad v\in\mdef(\A).
  \end{equation}
  In particular $\lambda\not\in\sigmaapp(\A)$, i.e.,
  $U\cap\sigmaapp(\A)=\varnothing$.
  The adjoint of $\A$ is the block operator matrix
  \[\A^*=\pmat{A^*&B\\B^*&D^*},\]
  which also satisfies the assumptions of this lemma.
  Indeed, \eqref{eq:D-uaccret} obviously also holds for $D^*$.
  Moreover,
  the uniform accretivity \eqref{eq:A-uaccret} of $A$ together with $0\in\varrho(A)$
  imply that $A-\gamma_A$ is
  \emph{m-accretive}, see \cite[\S V.3.10]{Kato95}.
  This in turn yields that $A^*-\gamma_A$ is m-accretive
  too and hence
  \[\Real\iprod{A^*x}{x}\geq\gamma_A\|x\|^2,\qquad x\in\mdef(A^*).\]
  It follows that \eqref{eq:bom-reg} also holds for $\A^*$.
  In particular $\ker\A^*=\{0\}$ or, equivalently, $\range(\A)\subset H\times H$
  is dense.
  On the other hand, \eqref{eq:bom-reg}
  implies that $\ker \A=\{0\}$ and that $\range(\A)$ is closed.
  Consequently $\range(\A)=H\times H$ and therefore $0\in\varrho(\A)$.
  Using  $\partial\sigma(\A)\subset\sigmaapp(\A)$
  and the connectedness of the set $U$,
  we obtain $U\subset\varrho(\A)$. Now \eqref{eq:bom-reg} implies
  $\|(\A-\lambda)^{-1}\|\leq 1/c_\lambda$ for all $\lambda\in U$.
\end{proof}

We consider approximations $\A_n$ of $\A$ of the  form
\[\A_n=\pmat{A_n&B_n\\B_n^*&D_n}\]
where
\begin{enumerate}
\item
  $(P_n,A_n)_{n\in\N}$ is a family which approximates $A$ strongly in the sense
  of Section~\ref{sec:strappr};
\item  all projections $P_n$
  are orthogonal and all $A_n$ are uniformly accretive with the same
  constant $\gamma_A$ as in \eqref{eq:A-uaccret};
\item
  $B_n=P_nB|_{U_n}$, $D_n=P_nD|_{U_n}$ where $U_n=\range(P_n)$
\end{enumerate}

\begin{lemma}\label{lem:bomapprox}
  \begin{enumerate}
  \item
    \(\set{\lambda\in\C}{-\gamma_D<\Real\lambda<\gamma_A}\subset\varrho(\A_n)\)
    and
    \[\|(\A_n-\lambda)^{-1}\|\leq\frac{1}{\min\{\gamma_A-\Real\lambda,
      \gamma_D+\Real\lambda\}},\quad -\gamma_D<\Real\lambda<\gamma_A.\]
  \item
    $(\cP_n,\A_n)_{n\in\N}$ approximates $\A$ strongly where
    $\cP_n=\diag(P_n,P_n)$.
  \end{enumerate}
\end{lemma}
\begin{proof}
  \begin{enumerate}
  \item
    From
    \[\iprod{D_nx}{x}=\iprod{P_nDx}{x}=\iprod{Dx}{x},\quad x\in U_n,\]
    it follows that $-D_n$ is uniformly accretive with constant $\gamma_D$
    from \eqref{eq:D-uaccret}.
    Consequently Lemma~\ref{lem:bomgap} can be applied to $\A_n$.
  \item
    In view of (a) and Lemma~\ref{lem:konsistenz} it suffices to show
    that for all $(x,y)\in\mdef(A)\times H$ there exist
    $(x_n,y_n)\in U_n\times U_n$ such that
    \begin{equation}\label{eq:bom-konsistenz}
      \lim_{n\to\infty}\pmat{x_n\\y_n}=\pmat{x\\y},\qquad
      \lim_{n\to\infty}\A_n\pmat{x_n\\y_n}=\A\pmat{x\\y}.
    \end{equation}
    Let $(x,y)\in\mdef(A)\times H$. From Lemma~\ref{lem:konsistenz}
    we get $x_n\in U_n$ with $x_n\to x$ and $A_nx_n\to Ax$ as $n\to\infty$.
    Set $y_n=P_ny$. Then $y_n\to y$ and
    \begin{align*}
      \|D_ny_n-Dy\|&\leq\|P_n(Dy_n-Dy)\|+\|P_nDy-Dy\|\\
      &\leq\|Dy_n-Dy\|+\|P_nDy-Dy\|\to 0,
      \quad n\to\infty,
    \end{align*}
    i.e., $D_ny_n\to Dy$. The proof of $B_ny_n\to By$ and $B_n^*x_n\to B^*x$
    is the same after the additional observation $B_n^*=P_nB^*|_{U_n}$.
    Hence we have shown \eqref{eq:bom-konsistenz}.\qedhere
  \end{enumerate}
\end{proof}

\begin{theo}
  Let
  \(s_1,\dots,s_m\in\set{\lambda\in\C}{-\gamma_D<\Real\lambda<\gamma_A}\).
  Let
  \lv{\(0<\eps < \min_{j=1,\dots,m}\left(\min\{\gamma_A-\Real s_j,
  \gamma_D+\Real s_j\}\right)\)}
  and
  \lv{
  \begin{align*}
  \delta_j>\frac{\eps}{\min\{\gamma_A-\Real s_j,
  \gamma_D+\Real s_j\}^2-\eps \min\{\gamma_A-\Real s_j,
  \gamma_D+\Real s_j\}}
  \end{align*}}
  for $j=1,\dots,m$. Then there exists $n_0\in\N$ such that
  \[\sigma_\eps(\A)\subset \bigcap_{j=1}^m
  \left[\left(B_{\delta_j}(W((\A_n-s_j)^{-1}))\right)^{-1}+s_j\right]
  \quad\text{for all}\quad n\geq n_0.\]
\end{theo}
\begin{proof}
  Lemma~\ref{lem:bomgap} and Lemma~\ref{lem:bomapprox} imply
  \lv{\[\|(\A-s_j)^{-1}\|\leq\frac{1}{\min\{\gamma_A-\Real s_j,
  \gamma_D+\Real s_j\}}\leq\frac{1}{\eps},\quad
  \|(\A_n-s_j)^{-1}\|\leq\frac{1}{\eps},\]}
  and hence the assertion follows from Theorem~\ref{theo:psincl-str-shift}.
\end{proof}

\begin{remark}
  Suppose that $A$ is the operator associated with a coercive sesqui-\\linear
  form $a$ on $V\subset H$ and that $U_n$, $W$, $P_n\in\mlin(H)$, $A_n\in\mlin(U_n)$
  are chosen as in Section~\ref{sec:fe-discr}.
  Then $(P_n,A_n)$ approximates $A$ uniformly, and hence also strongly,
  see Remark~\ref{rem:unifap}.
  Moreover, the coercivity of $a$ implies that $A$ and all $A_n$ are uniformly
  accretive with  constant $\gamma_A=\gamma$ from \eqref{eq:coercive}.
  Hence all assumptions of this section are fulfilled in this case.
\end{remark}

\section{Numerical examples}\label{sec:numericalexamples}
In order to exemplify the previously developed theory we take a look at the results of numerical computations. We investigate the steps that were involved in the discretization of a given operator and describe a visualization of supersets of the pseudospectrum.

\begin{example}
In this example we will examine the Hain-L\"ust operator which fits into the framework of section \ref{sec:discr-block}. See \cite{muh-mar12} and \cite{muh-mar13} for results on the approximation of the quadratic numerical range of such a block operator.
The Hain-L\"ust operator under consideration here is defined by
\begin{align*}
  \A = \pmat{
    A & B \\
    B^* & D
  }
\end{align*}
on the Hilbert space $L^2(0,1)\times L^2(0,1)$ where
$A=-\frac{1}{100}\frac{\partial^2}{\partial x^2} +2$, $B=
 I$ and $D = 2\e^{2\pi\iu\cdot}-3$ with $\D(A) =
\{u\in H^2(0,1) \mid u(0)=u(1)=0\}$, $\D(B)=\D(D)=L^2(0,1)$ and
$\D(\A) = \D(A)\oplus \D(D)$. Hence, for $u\in\D(\A)$ and $v\in
C^\infty(0,1)\times C^\infty(0,1)$ with $v(0)=v(1)=0$ we have
\lv{\begin{align}\label{eqn:sesquilinHL}
\iprod{\A u}{v} = &\int_0^1\left(\left(-\frac{1}{100}\frac{\partial^2}{\partial x^2} +2\right)u_1(x)+u_2(x)\right)\overline{v_1(x)}\,\mathrm{d}x \nonumber \\
& + \int_0^1\left( u_1(x) + \left(2\e^{2\pi\iu x}-3\right)u_2(x)\right)\overline{v_2(x)}\,\mathrm{d}x \nonumber \\
\begin{split}
= &\int_0^1\frac{1}{100}\frac{\partial}{\partial x}u_1(x)\frac{\partial}{\partial x}\overline{v_1(x)} + (2u_1(x)+u_2(x))\overline{v_1(x)}\,\mathrm{d}x\\
& + \int_0^1\left( u_1(x) + \left(2\e^{2\pi\iu x}-3\right)u_2(x)\right)\overline{v_2(x)}\,\mathrm{d}x.
\end{split}
\end{align}}
\begin{figure}
\includegraphics[width=\textwidth]{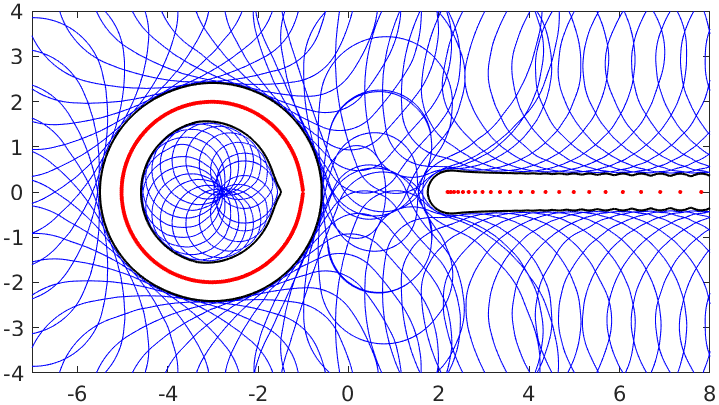}
\caption{Pseudospectrum approximation for the Hain-L\"ust operator}
\label{fig:plotHL}
\end{figure}
Let $\{\mathcal{T}_{\frac{1}{n}}\}_{n\in\N}$ be the family of decompositions of the interval $(0,1)$ where every subinterval $T\in \mathcal{T}_\frac{1}{n}$ is of length $\frac{1}{n}$ and let
\begin{align*}
U_n = \{u\in C(0,1) \mid u|_T \in \bbP_1(T), T\in \mathcal{T}_{\frac{1}{n}}, u(0)=u(1)=0\},\quad n\in\N.
\end{align*}
Here $\bbP_1(T)$ denotes the set of polynomials of degree 1 on the subinterval $T$. The piecewise linear functions
\begin{align*}
\widetilde{\varphi}_i = \begin{cases}
nx-i+1, \qquad & x\in(\frac{i-1}{n},\frac{i}{n}),\\
i+1-nx, \qquad & x\in(\frac{i}{n},\frac{i+1}{n}),\\
0, \qquad & \text{else},
\end{cases}
\end{align*}
for $i\in\{1,\dots,n-1\}$ form a basis of $U_n$ and therefore the functions
\begin{align*}
\varphi_i = \begin{cases}
(\widetilde{\varphi}_i,0), \qquad &i\leq n-1,\\
(0,\widetilde{\varphi}_{i-n+1}), \qquad &i>n-1,
\end{cases}
\end{align*}
for $i\in\{1,\dots,2(n-1)\}$ form a basis of $U_n\times U_n$. Evaluating \eqref{eqn:sesquilinHL} on these basis functions, the finite-element discretization matrices $\A_n$ of $\A$ are given by
\begin{align*}
\A_n = \left(\left(\iprod{\A\varphi_i}{\varphi_j}\right)_{i,j}\cdot\left(\iprod{\varphi_i}{\varphi_j}\right)_{i,j}^{-1}\right)^\intercal.
\end{align*}
Due to Lemma \ref{lem:bomapprox}, Theorem \ref{theo:psincl-str-shift} can be applied here. In order to illustrate the inclusion specified therein the boundaries of the sets \[\bigl(B_{\delta_j}(W((\A_n-s_j)^{-1}))\bigr)^{-1}+s_j\] (blue) are depicted in Figure \ref{fig:plotHL} for shifts $s_1,\dots,s_m\in\varrho(\A)$. The choice of the shifts was determined by the expected shape of the pseudospectrum aiming to obtain a relatively small superset thereof. They are located on two circles around $-3$ with radii greater and smaller than $2$ and on lines parallel to the real axis in the right half plane. Here \lv{$n=600$, $\delta_j=1.1\frac{\|(\A_n-s_j)^{-1}\|^2\eps}{1-\|(\A_n-s_j)^{-1}\|\eps}$} and \lv{$\eps\approx 0.4$}. The red dots are the eigenvalues of $\A_n$ while the black lines correspond to the boundaries of the pseudospectrum of the approximation matrix $\sigma_\eps(\A_n)$ computed by eigtool, see \cite{eigtool}. Note  that according to Theorem \ref{theo:psincl-str-shift} the intersection of the blue areas form an enclosure of the pseudospectrum of the actual operator $\sigma_\eps(\A)$, while the black lines only give the information for the
discretized operator.
Furthermore the spectral gap mentioned in Lemma \ref{lem:bomgap} becomes visible.
\end{example}

\begin{example}
\begin{figure}
\includegraphics[width=\textwidth]{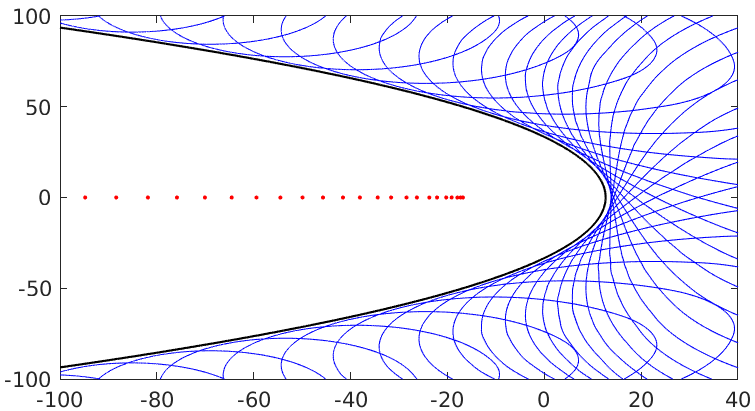}
\caption{Pseudospectrum approximation for the advection-diffusion operator}
\label{fig:plotAD}
\end{figure}
Let us consider the the advection-diffusion operator $A:\D(A)\subset L^2(0,1)\rightarrow L^2(0,1)$ defined by
\begin{align*}
A = \eta\frac{\partial^2}{\partial x^2} + \frac{\partial}{\partial x}
\end{align*}
with $\D(A) = \{u\in H^2(0,1) \mid u(0)=u(1)=0\}$, which has also been examined in \cite[pp. 115]{ETBook}. For $u\in \D(A)$ and $v\in C^\infty(0,1)$ we have
\lv{\begin{align}\label{eqn:sesquilinAD}
\langle Au,v\rangle &= \int_0^1 \left(\eta\frac{\partial^2}{\partial x^2}u(x) + \frac{\partial}{\partial x}u(x)\right)\overline{v(x)}\,\mathrm{d}x \nonumber \\
&= \int_0^1 \frac{\partial}{\partial x}u(x)\overline{v(x)} - \eta\frac{\partial}{\partial x}u(x)\frac{\partial}{\partial x}\overline{v(x)}\,\mathrm{d}x.
\end{align}}
As in the previous example let $\{\mathcal{T}_{\frac{1}{n}}\}_{n\in\N}$ be the family of decompositions of the interval $(0,1)$ where every subinterval $T\in \mathcal{T}_\frac{1}{n}$ is of length $\frac{1}{n}$ and let
\begin{align*}
U_n = \{u\in C(0,1) \mid u|_T \in \mathbb P_1(T), T\in \mathcal{T}_{\frac{1}{n}}, u(0)=u(1)=0\},\quad n\in\mathbb N.
\end{align*}
Here $\mathbb P_1(T)$ denotes the set of polynomials of degree 1 on the subinterval $T$. The piecewise linear functions
\begin{align*}
\varphi_i = \begin{cases}
nx-i+1, \qquad & x\in(\frac{i-1}{n},\frac{i}{n}),\\
i+1-nx, \qquad & x\in(\frac{i}{n},\frac{i+1}{n}),\\
0, \qquad & \text{else},
\end{cases}
\end{align*}
for $i\in\{1,\dots,n-1\}$ form a basis of $U_n$. Evaluating \eqref{eqn:sesquilinAD} on these basis functions, the finite-element discretization matrices $A_n$ of $A$ are given by
\begin{align*}
A_n = \left(\left(\langle A\varphi_i,\varphi_j\rangle\right)_{i,j}\cdot\left(\langle \varphi_i,\varphi_j\rangle\right)_{i,j}^{-1}\right)^\intercal.
\end{align*}
With the choice of $\eta = 0.015$, Figure \ref{fig:plotAD} shows the eigenvalues of $A_n$ for $n=40$ (red) and the sets
\[\bigl(B_{\delta_j}(W((A_n-s_j)^{-1}))\bigr)^{-1}+s_j\]
(blue) for a number of shifts $s_1,\dots,s_m$ where \lv{$\delta_j=1.1\frac{\|(A_n-s_j)^{-1}\|^2\eps}{1-\|(A_n-s_j)^{-1}\|\eps}$} and \lv{$\eps\approx 16$}. The shifts are located at a certain distance to the expected pseudospectrum so as to obtain a relatively small superset thereof. The black line corresponds to the boundary of $\sigma_\eps(A_n)$ computed by eigtool, see \cite{eigtool}.
This demonstrates the result of Theorem \ref{theo:psincl-str-shift} which actually yields an enclosure for the pseudospectrum of the operator $A$ while the black line  only shows the boundary of the pseudospectrum of the approximation matrix $A_n$.
\end{example}

\lv{
As already mentioned in Remark \ref{rem:NumericalRange} we also have the enclosure
\begin{align*}
\sigma_\eps(A)\subset B_\eps(W(A))
\end{align*}
for operators $A$ with a compact resolvent.
\cw{Note that, because both sides of the enclosure are
  in terms of the same operator $A$,
  this only yields an enclosure for the discretized operator
   when applied numerically, not the full operator.}
So let us
take a look at the discretizations of the Hain-L\"ust (Figure
\ref{fig:plotHL_NR}) and the advection-diffusion operator (Figure
\ref{fig:plotAD_NR}) again.
\begin{figure}
\includegraphics[width=\textwidth]{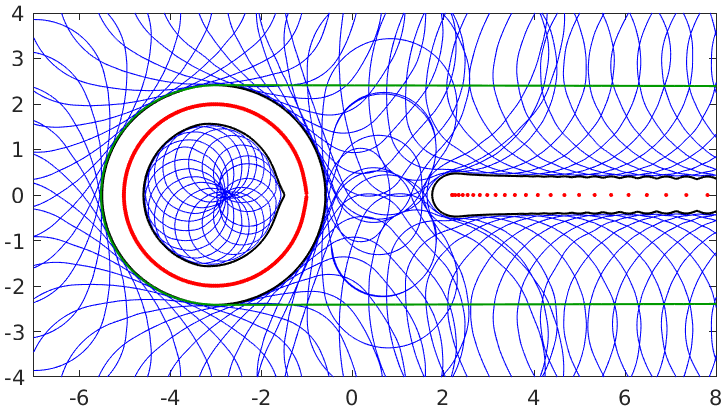}
\caption{$\eps$-neighborhood of the numerical range of the Hain-L\"ust operator}
\label{fig:plotHL_NR}
\end{figure}
Here, the $\eps$-neighborhoods of the numerical ranges are depicted by green lines. As you can see, this approach leads to a very similar result in case of the advection-diffusion operator (where the pseudospectrum is convex), while it fails to distinguish disconnected components of the pseudospectrum in case of the Hain-L\"ust operator.
\begin{figure}\centering
\includegraphics[width=\textwidth]{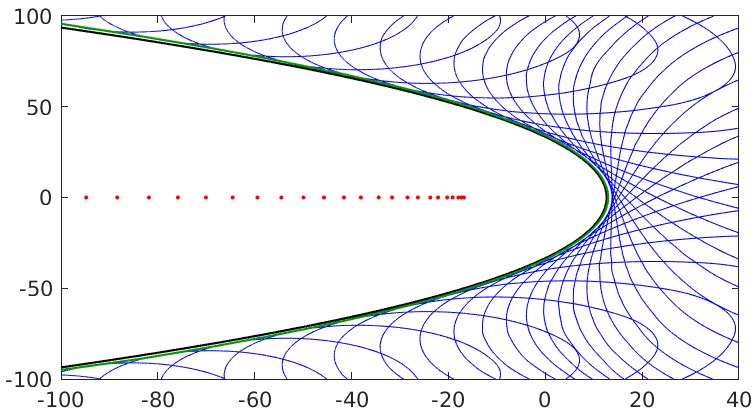}
\caption{$\eps$-neighborhood of the numerical range of the advection-diffusion operator}
\label{fig:plotAD_NR}
\end{figure}
}

\section*{Acknowledgment}
\cw{%
  We thank the anonymous referee for thoroughly reading the manuscript
  and his/her valuable remarks and suggestions.
The referee in particular pointed out an improvement
in Proposition~\ref{prop:psincl} and raised the question of optimality of the
enclosure \eqref{eq:psincl-shift}.
Ultimately, this led to
the results of Theorem~\ref{theo:optimality}.
}

\bibliographystyle{abbrv}
\bibliography{pseudo}

\end{document}